\documentclass[10pt]{amsart}
\usepackage{amsmath}
\usepackage{amsfonts}
\usepackage{amssymb}
\usepackage{amsthm}
\usepackage{mathrsfs}       
\usepackage{graphicx}     
\usepackage[utf8]{inputenc}
\allowdisplaybreaks

\usepackage{xspace}
\usepackage[colorlinks=true]{hyperref}

\newcommand{\R}{\mathbb{R}}
\newcommand{\N}{\mathbb{N}}
\newcommand{\T}{\mathbb{T}}
\newcommand{\Z}{\mathbb{Z}}

\newcommand{\C}{\mathscr{C}}

\newcommand{\V}{\mathcal{V}}
\newcommand{\A}{\mathcal{A}}
\newcommand{\Y}{\mathcal{Y}}
\newcommand{\Hdiv}{\mathcal{H}}

\newcommand{\ud}{\,\mathrm{d}}

\renewcommand{\d}{\partial}

\renewcommand{\S}{\mathscr S}

\let \div \relax
\DeclareMathOperator{\div}{div}

\DeclareMathOperator{\dist}{dist}

\DeclareMathOperator{\supp}{supp}

\newcommand{\ovl}[1]{\overline{#1}}

\newtheorem{thm}{THEOREM}[section]
\newtheorem{remark}[thm]{REMARK}
\newtheorem{lemma}[thm]{LEMMA}
\newtheorem{definition}[thm]{DEFINITION}
\newtheorem{proposition}[thm]{PROPOSITION}
\newtheorem{corollary}[thm]{COROLLARY}

\newcounter{thmbiss}

\title[Controllability of the 2-D Vlasov-Navier-Stokes system]{Local null-controllability of the 2-D Vlasov-Navier-Stokes system}

 \author[I.~Moyano]{Iv\'an Moyano}
\begin{address}{Iv\'an Moyano,
Centre de math\'ematiques Laurent Schwartz, Ecole polytechnique, CNRS, Universit\'e Paris-Saclay, 91128 Palaiseau, France}
  \email{\href{mailto:ivan.moyano@math.polytechnique.fr}{ivan.moyano-garcia@.polytechnique.edu}}
 \end{address}
\thanks{{\em Acknowledgment:}
The author would like to very much thank Daniel Han-Kwan (CNRS and CMLS, Ecole Polytechnique) for suggesting him this problem and for many fruitful discussions and advices.
}

\begin{document}
\maketitle

\begin{abstract}
We prove a null controllability result for the Vlasov-Navier-Stokes system, which describes the interaction of a large cloud of particles immersed in a fluid. We show that one can modify both the distribution of particles and the velocity field of the fluid from any initial state to the zero steady state, by means of an internal control. Indeed, we can modify the non-linear dynamics of the system in order to absorb the particles and let the fluid at rest. The proof is achieved thanks to the return method and a Leray-Schauder fixed-point argument. \\  

\end{abstract}

\tableofcontents

\section{Introduction}
The Vlasov-Navier-Stokes system describes the behaviour of a large cloud of particles immersed in a viscous incompressible fluid. The interaction of the particles with the surrounding fluid is taken into account through a coupling between a Vlasov equation, modelling the transport of the particles, and the Navier-Stokes system, governing the evolution of the fluid (see Section \ref{sec: reviewVNS} for more details). In this article, we are interested in the controllability of the dynamics of both the particles and the fluid by means of a control located in a subset of the phase space. \par 

More precisely, we consider the Vlasov-Navier-Stokes system in the 2-dimensional torus $\T^2 := \R^2 / \Z^2$, which writes, for $T>0$ and $\omega \subset \T^2$, 
\begin{equation}
\left\{  \begin{array}{ll}
\partial_t f + v\cdot \nabla_x f + \div_v\left[ (u-v)f  \right]= 1_{\omega}(x)G, & (t,x,v) \in (0,T) \times \T^2 \times \R^2, \\
\partial_t u + \left( u\cdot \nabla_x\right) u - \Delta_x u + \nabla_x p = j_f - \rho_f u, & (t,x)\in (0,T) \times \T^2, \\
\div_x u(t,x) = 0, & (t,x) \in (0,T) \times \T^2, \\
f_{|t=0}=f_0(x,v), & (x,v) \in \T^2 \times \R^2,  \\
u_{|t=0}= u_0(x), & x\in \T^2, 
\end{array} \right.
\label{eq:VNS}
\end{equation} where 
\begin{equation}
j_f(t,x) := \int_{\R^2} vf(t,x,v) \ud v, \quad \rho_f(t,x):= \int_{\R^2} f(t,x,v) \ud v.
\label{eq:currentdendisty}
\end{equation} We shall suppose throughout the article that the Lebesgue measure of the torus is normalised, i.e., $\int_{\T^2} \ud x =1$. \par

\subsection{Main results}
Before stating our main results, we give the notion of solution that we use in this work and we explain the controllability problem that we want to solve.

\subsubsection{Strong solutions}
\begin{definition}
Let $T>0$. Let $f_0 \in \C^1(\T^2 \times \R^2)$ and $u_0 \in H^1(\T^2;\R^2)$ with $\div_x u_0 = 0$. Let $G\in \C^0([0,T] \times \T^2 \times \R^2)$. We say that $(f,u)$ is a strong solution of system (\ref{eq:VNS}) if the following conditions are satisfied:
\begin{align}
& f\in \C^1([0,T]\times \T^2 \times \R^2), \label{eq:regularitytranport} \\
& \textrm{the Vlasov equation is satisfied for every } (t,x,v) \in (0,T) \times \T^2 \times \R^2, \label{eq:Vlasovclassical} \\
& \sup_{t\in [0,T]} \int_{\T^2} \int_{\R^2} \left( 1 + |v| +  |v|^2 \right) f(t,x,v) \ud x \ud v < \infty, \label{eq:moment2borne} \\
& u \in \C^0([0,T];H^1(\T^2;\R^2)) \cap L^2(0,T;H^2(\T^2;\R^2)), \label{eq:champregulier} \\
& \div_x u(t,x)=0, \quad \forall t\in [0,T],
\end{align} and for any $\psi \in \C^1([0,T];H^1(\T^2;\R^2))$ with $\div_x \psi(t,x)=0$ and $t\in (0,T]$, one has 
\begin{align}
& \int_{\T^2} u(t)\psi(t) \ud x + \int_0^t \int_{\T^2} \left( \nabla u : \nabla \psi - u \otimes u \cdot \nabla \psi -u\partial_t \psi \right) \ud s \ud x \nonumber \\
& \quad \quad \quad \quad = \int_{\T^2} u_0 \psi(0) \ud x + \int_0^t \int_{\T^2} \left( j_f(s) - \rho_f(s) u(s)  \right) \psi(s) \ud s \ud x, \label{eq:formulationfaibleVNS}
\end{align} where 
\begin{equation*}
\nabla u : \nabla \psi:= \sum_{j,k=1}^2 \partial_j u^k \partial_j \psi^k, \quad u \otimes u \cdot \nabla \psi := \sum_{j,k=1}^2 u^j u^k \partial_j \psi^k. 
\end{equation*} Let us recall that, under the incompressibility condition, the convection term satisfies that $(u\cdot \nabla) u = \div(u \otimes u)$, with the previous notation.
\label{definition:strong}
\end{definition}

\subsubsection{The controllability problem}

We are interested in the controllability properties of system (\ref{eq:VNS}), by means of an internal control, in the following sense. Given $f_0$ and $f_1$ in a suitable function space and given $T>0$, is it possible to find a control $G$ steering the solution of (\ref{eq:VNS}) from $f_0$ to $f_1$, in time $T$? In other words, we want to find $G$ such that
\begin{equation*}
f(T,x,v) = f_1(x,v), \quad \forall \,  (x,v) \in  \T^2  \times \R^2.
\end{equation*} 
Another interesting point is to find $G$ in such a away it could modify not only the dynamics of the distribution function $G$ but also the evolution of the field $u$ from $u_0$ to a given $u_1$ in time $T$. \par  

In this article, we shall give a positive answer to this question for $f_1 = 0$ and $u_1 =0$. We need a geometric assumption on the region $\omega$, stated in \cite{Glass} and used in \cite[Definition 1.2, p.699]{GDHK2}.

\begin{definition}[Strip assumption] An open set $\omega \subset \T^2$ satisfies the strip assumption if it there exists a straight line of $\R^2$ whose image $H$ by the canonical surjection $s:\R^2 \rightarrow \T^2$ is closed and included in $\omega$. We shall call $n_H$ a unit vector orthogonal to $H$, in such a way $H=\left\{ x\in \R^2; x\cdot n_H = 0  \right\}$. For any $l>0$, we denote 
\begin{equation*}
H_l := H + [-l,l]n_H.
\end{equation*} Let us observe that, as $H$ is closed in $\T^2$, there exists $\delta>0$ such that 
\begin{equation}
H_{2\delta} \subset \omega
\label{eq:hyperplan}
\end{equation} and such that $4\delta$ is smaller than the distance between two successive lines in $s(H)$.
\label{definition:strip}
\end{definition}

Under this geometric assumption and suitable hypothesis on the data $u_0$, $f_0$ and $f_1$, we obtain the following local null-controllability result in large time.

\begin{thm}
Let $\gamma>2$, and let $\omega \subset \T^2$ satisfy the strip assumption of Definition \ref{definition:strip}. There exists $\epsilon>0$, $M>0$ and $T_0 >0$ such that for every $T\geq T_0$, $f_0 \in \C^1(\T^2\times \R^2) \cap W^{1,\infty}(\T^2 \times \R^2)$ and $u_0$ satisfying that 
\begin{equation}
u_0 \in \C^1(\T^2;\R^2) \cap H^2(\T^2;\R^2), \quad \quad \div_x u_0 = 0 , \quad \quad \|u_0\|_{H^{\frac{1}{2}}(\T^2)} \leq M, 
\label{eq:massechampinitiale} 
\end{equation} and that 
\begin{align}
&\|f_0\|_{\C^1(\T^2\times\R^2)} + \| (1+|v|)^{\gamma+2} f_0\|_{\C^0(\T^2\times \R^2)} \leq \epsilon, \label{eq:smalldata} \\
&\exists \kappa>0, \quad \sup_{\T^2 \times \R^2} (1+|v|)^{\gamma} \left(|\nabla_x f_0| + |\nabla_v f_0| \right)(x,v) \leq \kappa, \label{eq:conditionforuniqueness}  
\end{align} there exists a control $G \in \C^0([0,T]\times \T^2 \times \R^2)$ such that a strong solution of (\ref{eq:VNS}) with $f|_{t=0} = f_0$ and $u|_{t=0} = u_0$ exists, is unique and satisfies \begin{equation}
f|_{t=T} = 0, \quad u|_{t=T}=0.
\label{eq:nullcontrollability}
\end{equation}
\label{thm:controllability}  
\end{thm} The techniques developed in this article allow to obtain, in absence of the control term $1_{\omega}(x) G$, that the strong solutions of the homogeneous Vlasov-Navier-Stokes
  
  \begin{equation}
\left\{  \begin{array}{ll}
\partial_t f + v\cdot \nabla_x f + \div_v\left[ (u-v)f  \right] = 0, & (t,x,v) \in (0,T) \times \T^2 \times \R^2, \\
\partial_t u + \left( u\cdot \nabla_x\right) u - \Delta_x u + \nabla_x p = j_f - \rho_f u, & (t,x)\in (0,T) \times \T^2, \\
\div_x u(t,x) = 0, & (t,x) \in (0,T) \times \T^2, \\
f_{|t=0}=f_0(x,v), & (x,v) \in \T^2 \times \R^2,  \\
u_{|t=0}= u_0(x), & x\in \T^2, 
\end{array} \right.
\label{eq:VNShomogeneous}
\end{equation} are unique within a certain class. The result is the following.

\begin{thm}
Let $\gamma>2$ and $T>0$. Then, for any $M>0$, there exists $\epsilon>0$ such that for every $f_0 \in \C^1(\T^2\times \R^2) \cap W^{1,\infty}(\T^2 \times \R^2)$ and $u_0$ satisfying that 
\begin{equation*}
u_0 \in \C^1(\T^2;\R^2) \cap H^2(\T^2;\R^2), \quad \quad \div_x u_0 = 0 , \quad \quad \|u_0\|_{H^{\frac{1}{2}}(\T^2)} \leq M, 
\end{equation*} and that 
\begin{align}
&\|f_0\|_{\C^1(\T^2\times\R^2)} + \| (1+|v|)^{\gamma+2} f_0\|_{\C^0(\T^2\times \R^2)} \leq \epsilon, \label{eq:smalldatawp} \\
&\exists \kappa>0, \quad \sup_{\T^2 \times \R^2} (1+|v|)^{\gamma} \left(|\nabla_x f_0| + |\nabla_v f_0| \right)(x,v) \leq \kappa, \label{eq:conditionforuniquenesswp}  
\end{align} there exists a unique strong solution of (\ref{eq:VNShomogeneous}) with $f|_{t=0} = f_0$ and $u|_{t=0} = u_0$.
\label{thm:WP}  
\end{thm}

\begin{remark}
Theorem \ref{thm:controllability} is a local null-controllibility result in the sense that the smallness conditions (\ref{eq:smalldata}) and (\ref{eq:massechampinitiale}) are essential to prove the controllability. On the other hand, condition (\ref{eq:conditionforuniqueness}) is useful to obtain certain stability estimates (see Section \ref{sec: stability VNS}), that are key to prove the controllability of the system and the uniqueness of the corresponding solution.  \par 

Theorem \ref{thm:WP} shows the existence of strong solutions of (\ref{eq:VNShomogeneous}) with small-data in any time. Condition (\ref{eq:conditionforuniquenesswp}) in this case ensures that this strong solution is unique. 
\end{remark}

\subsection{Previous work}
\subsubsection{The controllability of non-linear kinetic equations}

The controllability of non-linear equations, typically described by the coupling of a Vlasov equation and a system for a vector field, originated from the work by O. Glass on the Vlasov-Poisson system in \cite{Glass}. In this work, the idea of combining the return method (see Section \ref{sec: retourVNS} for details) with a Leray-Schauder fixed-point argument involving an absorption procedure was successfully employed for the first time.  \par

This strategy was later extended in \cite{GDHK1} by O. Glass and D. Han-Kwan to the Vlasov-Poisson system under external and Lorentz forces. The authors obtain both local and global exact controllability results in the case of bounded external forces, which requires some new ideas to construct the reference trajectories. Precisely, the authors exploit the fact that the free dynamics and the dynamics under the external force are similar in small time. In the case of Lorentz forces, a precise knowledge of the magnetic field and a geometric control condition in the spirit of \cite{BLR} allow to obtain a local exact controllability result. The functional framework of \cite{Glass, GDHK1} is the one given by the classical solution of the Vlasov-Poisson system, that is, some appropriate H\"{o}lder spaces, according to \cite{UkaiOkabe}. To end up, let us mention that the systems considered in these results present a coupling with a Poisson equation, which is stationary, allowing the use of techniques from Harmonic approximation to construct the reference trajectories. \par

In the case in which the Vlasov equation is coupled with a non-stationary equation, the construction of a reference trajectory has to be achieved in a different way. In this direction, a new strategy has been developed by O. Glass and D. Han-Kwan in \cite{GDHK2} in the context of the Vlasov-Maxwell system. In this case, the authors use some controllability results for the Maxwell system, under the geometric control condition of \cite{BLR}, which allows to construct suitable reference trajectories. In a second step, the Leray-Schauder fixed-point procedure must be reformulated in order to respect some conservation laws. This gives a local controllability result for the distribution function. This results holds in some appropriate Sobolev spaces, according to the functional framework of \cite{Asano, Wollman}. \par 

Furthermore, their strategy allows to obtain a local controllability result for the distribution function under the assumption that $\omega$ contains a hyperplane, using the convergence towards the Vlasov-Poisson system under a certain regime.  \par 

Finally, the methods of \cite{Glass} and \cite{GDHK1} have been applied by the author to a kinetic-fluid system in \cite{VS}, the Vlasov-Stokes system, where a Vlasov equation is coupled with a stationary Stokes system, which can be seen as a simplified version of (\ref{eq:VNS}).

\subsubsection{A short review on the Vlasov-Navier-Stokes system}
\label{sec: reviewVNS}
This system is a model to describe the behaviour of a large cloud of particles interacting with a viscous incompressible fluid. Typically, the coupling is made through two mechanisms. The action of the fluid on the particles is taken into account in the Vlasov equation, where the field appears multiplying the gradient in velocity of the distribution function. Secondly, the action of the particles on the fluid gives rise to a drag force appearing in the Navier-Stokes system as a source term. For more details on the model, we refer to \cite{Boudin}.\par  

The field equation in system (\ref{eq:VNS}), under the influence of the drag force $j-\rho u$, has been rigorously derived as a mean-field limit of a large cloud of particles by L. Desvillettes, F. Golse and V. Ricci in \cite{DesvillettesGolse}, using homogeneisation techniques and under a strong non-collision hypothesis. \par 

The well-posedness of a simplified system has been done in \cite{Hamdache}. The existence of weak solutions to (\ref{eq:VNS}) in the three-dimensional torus has been achieved by L. Boudin, L. Desvillettes, C. Grandmont and A. Moussa in \cite{BoudinDesvillettes}. Other related systems, considering variable density or compressible fluids have been studied by Y.-P. Choi and B. Kwon in \cite{ChoiKwon,Choi}, along with its asymptotic behaviour.  \par

Finally, the question of hydrodynamical limits under certain regimes has been treated by T. Goudon, P. E. Jabin and A. Vasseur in \cite{Goudon1,Goudon2}, considering also the effects of collisions between particles.  

\subsubsection{Obstructions to controllability}

Since Theorem \ref{thm:controllability} is a result of local nature around the steady state $(f,u)=(0,0)$, a first step to achieve its proof could be the use of the linear test (see \cite{Coron}). Following the classical scheme, the controllability of the linearised system around the trivial trajectory and the classical inverse mapping theorem between proper functional spaces would imply the controllability of the nonlinear system (\ref{eq:VNS}). \par 
Indeed, the formal linearised equation around the trajectory $(f,u) = (0,0)$ is
\begin{equation}
\left\{ \begin{array}{ll}
\d_t F + v\cdot \nabla_x F - v\cdot \nabla_v F - 2 F = 1_{\omega}(x)\tilde{G}, \\
F(0,x,v) = f_0(x,v),
\end{array} \right.
\label{eq:linearised}
\end{equation} which is a transport equation with friction. By the method of characteristics, we can give an explicit solution of (\ref{eq:linearised}), which writes
\begin{equation}
F(t,x,v) = e^{2t}f_0(x+ (1-e^{t})v, e^t v) + \int_0^t e^{2(t-s)} (1_{\omega}\tilde{G})(s,x+(1-e^{t-s})v, e^{t-s}v) \ud s.
\label{eq:explicit}
\end{equation} As pointed out in \cite{Glass}, there exist two obstructions for controllability, which are:
\begin{description}
\item[Small velocities] a certain $(x,v)\in \T^2 \times \R^2$ can have a "good direction" with respect to the control region $\omega$, in the sense that $x+(1-e^{-t})v$ meets $\omega$ at some time. However, if $|v|$ is not sufficiently large, the trajectory of the characteristic beginning at this point would possibly not reach $\omega$ before a fixed time. In our case, the effects of friction could enhance this difficulty.
\item[Large velocities] the obstruction concerning large velocities is of geometrical nature. There exist some "bad directions" with respect to $\omega$, in the sense that a characteristic curve beginning at $(x,v)\in \T^2\times \R^2$ would never reach $\omega$, no matter how large $|v|$ is. 
\end{description}  As a result of this, and considering again equation (\ref{eq:explicit}), we deduce that the linearised system is not controllable in general.

\subsubsection{The return method}
\label{sec: retourVNS}
In order to circumvent these difficulties, we use the return method, due to J.-M. Coron.  \par
The idea of this method, in the case under study, is to construct a reference trajectory $(\ovl{f}, \ovl{u})$ starting from $(0,0)$ and coming back to $(0,0)$ at some fixed time in such a way the linearised system around it is controllable. This method, which makes a crucial use of the nonlinearity of the system, allows to avoid the obstructions discussed in the previous section. \par 
We refer to \cite{Coron, GlassBourbaki} for presentations and examples on the return method.

\subsubsection{Strategy of the proof of Theorem \ref{thm:controllability}}

The strategy of this work follows the scheme of \cite{GDHK2}. More precisely, it relies on two ingredients, combined with a final step needed to reach the zero state.
\begin{description}

\item[Step 1] We build a reference solution $(\ovl{f},\ovl{u})$ of system (\ref{eq:VNS}) with a control $\ovl{G}$, located in $\omega$, starting from $(0,0)$ and arriving at $(0,0)$ at a sufficiently large time $T>0$ and such that the characteristics associated to the field $-v + \ovl{u}$ meet $\omega$ before $T>0$. In doing this, the non-linear coupling will be essential, thanks to the use of controllability results for the Navier-Stokes system.

\item[Step 2] We build a solution $(f,u)$ close to $(\ovl{f},\ovl{u})$ starting from $(f_0,u_0)$ and such that 
\begin{equation*} 
f|_{t=T} = 0  \textrm{ outside } \omega, 
\end{equation*} that is, all the particles are confined in $\omega$ at time $T>0$. This can be done by means of a fixed-point argument involving an absorption operator in the control region $\omega$. 

\item[Step 3] We modify the distribution function inside $\omega$ in order to get the zero distribution at some time $T+\tau_1$, i.e, $f|_{t=T+\tau_1}=0$. We then modify the velocity field thanks to the coupling term and a controllability result for the Navier-Stokes system, which yields
\begin{equation*}
(f,u)|_{t=T+\tau_1 + \tau_2} = (0,0),
\end{equation*} for some $\tau_2 > 0$. \par 
\vspace{0.5em}
Let us note that the article centres mainly in the proof of Theorem \ref{thm:controllability}, as the proof of Theorem \ref{thm:WP} follows from some minor modifications. This will be clear from the proofs. 

\end{description}

\subsubsection{Outline of the paper}
In Section \ref{sec: characteristicsVNS}, we recall some results on the characteristic equations that will be important in the sequel. In Section \ref{sec: referenceVNS}, we construct a suitable reference trajectory of system (\ref{eq:VNS}). In Section \ref{sec: fixedpointVNS}, we construct a strong solution of this system, thanks to a fixed-point argument. In Section \ref{sec: stability VNS} we prove some stability estimates for the Navier-Stokes system. In Section \ref{sec: fixedpointrelevantVNS}, we show that this strong solution satisfies the controllability property. In Section \ref{sec: uniquenessVNS}, we prove that this strong solution is unique within a certain class, which concludes the proof of Theorem \ref{thm:controllability}. In Section \ref{sec: commentsVNS} we gather some comments and perspectives. In Appendices \ref{sec: AppendixS} and \ref{sec: AppendixNS} we recall some results on the Stokes and Navier-Stokes systems.

\subsubsection{Notation and functional framework}
\label{sec: notationVNS}
Let $T>0$ and set $Q_T:=[0,T] \times \T^2 \times \R^2$ and $\Omega_T := [0,T] \times \T^2$. If $\sigma \in [0,1]$, $\C^{0,\sigma}_b(\Omega)$ denotes the space of bounded $\sigma-$H\"{o}lder functions in $Q_T$, equipped with the norm
\begin{equation}
\| f \|_{\C^{0,\sigma}_b(\Omega)} := \|f\|_{L^{\infty}(\Omega)} + \sup_{(t,x,v) \not = (t',x',v')} \frac{|f(t,x,v) - f(t',x',v')|}{|(t,x,v) - (t',x',v')|^{\sigma}},
\end{equation} for any $f\in \C^{0,\sigma}_b(Q_T)$. We shall also consider the spaces $\C^{0,\sigma}_b(\Omega_T)$, with analogous definitions. \par 
We will also use the Sobolev spaces $W^{m,p}$, with $m\in \N^*$ and $p\in [1,\infty]$. In the particular case of the flat torus $\T^2$, the Fourier series allow to write
\begin{equation}
f = \sum_{k\in \Z^2} f_k e^{ik \cdot x}, \quad \textrm{ in }L^2(\T^2), \quad \forall f \in L^2(\T^2),
\label{eq:Fourierseries}
\end{equation} with
\begin{equation}
f_k := \int_{\T^2} f(x) e^{i k\cdot x} \ud x, \quad \forall k \in \Z^2.
\label{eq:FouriercoefficientsVNS}
\end{equation} Thus, for any $s>0$, we may write
\begin{align*}
&H^s(\T^2) = \left\{ f\in L^2; \, f=\sum_{k\in\Z^2} f_k e^{ik\cdot x}, \, \ovl{f_k} = f_{-k}, \, \sum_{k\in\Z^2} \left( 1 + |k|^2\right)^s |f_k|^2 < \infty \right\}, \\
& H^s_0(\T^2)  = \left\{ f\in H^s(\T^2); \, \int_{\T^2} f(x)\ud x = 0 \right\},
\end{align*} which allows to equip these spaces, respectively, with the norms
\begin{equation}
\|f \|_{H^s} := \left( \sum_{k\in\Z^2}(1 + |k|^2)^s|f_k|^2  \right)^{\frac{1}{2}}, \quad \|f \|_{H_0^s} := \left( \sum_{k\in\Z^2}|k|^{2s}|f_k|^2  \right)^{\frac{1}{2}},
\label{eq:normsFourier}
\end{equation} with equivalence of norms in the case of $H^s_0$ as a subspace of $H^s$.\par 
In the case of vector fields, we shall use $(W^{m,p}(\T^2))^2$, with the product norm. Let us introduce, as usual, the space of solenoidal vector fields in $L^2$, i.e.,
\begin{equation*}
\mathcal{H}:= \left\{ F \in L^2(\T^2)^2; \, \div_x F = 0 \textrm{ in } \R^2 \right\},
\end{equation*} where the operator $\div_x$ is taken in the distributional sense. Analogously, let us use the following notations, following \cite{Chemin},
\begin{align*}
\V_{\sigma} &:= \left\{ F \in H^1(\T^2)^2; \, \div_x F = 0 \textrm{ in } \R^2 \right\}, \\
\V_{\sigma}' &:= \left\{ F \in H^{-1}(\T^2)^2; \, \div_x F = 0 \textrm{ in } \R^2 \right\}, \\
\V' &:= \left\{ F \in H^{-1}(\T^2)^2\right\}.
\end{align*} We shall also denote by $\S(\R^2)$ the space of Schwartz functions in $\R^2$. \par 

Finally, if $X$ is a Banach space and $p \geq 1$, we will sometimes use, for simplicity, the notations $L^p_tX_x$ or $\C^0_tX_x$ to refer to $L^p(0,T;X)$ or $\C^0([0,T];X)$. \par 
To simplify some computations, we shall use the symbol $\lesssim$ to denote that a multiplicative constant is omitted. \par

\section{Some remarks on the characteristic equations}
\label{sec: characteristicsVNS}
Let be given a fixed $u(t,x)$. Let $s,t \in [0,T]$, $(x,v)\in \T^2\times \R^2$. We denote by $(X(t,s,x,v),V(t,s,x,v))$ the characteristics associated with the field $-v + u(t,x)$, i.e., the solution of the system
\begin{equation}
\left\{ \begin{array}{ll}
\frac{\ud}{\ud t} \begin{pmatrix} X \\ V \end{pmatrix} = \begin{pmatrix}   V(t)  \\ - V(t) + u(t,X) \end{pmatrix}, \\
\begin{pmatrix}   X \\ V \end{pmatrix}_{|t=s} = \begin{pmatrix} x \\ v \end{pmatrix}. 
\end{array} \right.
\label{eq:characteristcsystem}
\end{equation} We observe that if $u\in \C^0([0,T]; \C^1(\T^2;\R^2))$, system (\ref{eq:characteristcsystem}) has a unique solution, thanks to the Cauchy-Lipschitz theorem. Moreover, one has the explicit formulae 
\begin{equation}
\left\{ \begin{array}{ll}
X(t,s,x,v) = x+(1-e^{-t+s})v + \int_s^t \int_s^{t'} e^{\tau - t'} u(\tau, X(\tau,s,x,v)) \ud \tau \ud t', \\
V(t,s,x,v) = e^{-t+s}v + \int_s^t e^{\tau - t} u(\tau, X(\tau,s,x,v))\ud \tau.
\end{array} \right.
\label{eq:characteristics}
\end{equation} 

\begin{remark}
The same result is still valid if one considers vector field in the class $u\in L^1(0,T;\C^{0,1}(\T^2;\R^2))$. This will be important in Section \ref{sec: fixedpointVNS}. For details, see \cite[Remark 1.2.3]{Crippa}. In that case, the associated characteristics are still given by (\ref{eq:characteristics}), well-defined for every $t\in [0,T]$ and differentiable also in time.
\end{remark}

Using the method of characteristics, given an initial datum $f_0 \in \C^0(\T^2 \times \R^2)$, the solution of the transport equation with friction 
\begin{equation}
\left\{ \begin{array}{ll}
\d_t f + v\cdot \nabla_x f + \div_v\left[(u-v)f  \right] = 0, & (t,x,v)\in (0,T)\times \T^2 \times \R^2, \\
f(0,x,v) = f_0(x,v), & (x,v) \in \T^2 \times \R^2, 
\end{array} \right.
\end{equation} has the explicit solution
\begin{equation}
f(t,x,v) = e^{2t}f_0((X,V)(0,t,x,v)),
\label{eq:explicitsolutionVNS}
\end{equation} where $(X,V)$ are given by (\ref{eq:characteristics}). \par 
The proof of the following result, under the hypothesis that the field belongs to $\C^0_t \C^1_x$, can be found in \cite[Lemma 1, Section 3]{VS}. The adaptation to the case in which the field belongs to $L^1_t \C^{0,1}_x$ is straightforward.

\begin{lemma}
Let $u \in L^1(0,T; \C^{0,1} (\T^2;\R^2))$ Then, the characteristics associated to the field $-v + u$ satisfy that for some $C = C(T,\|u\|_{L^1_t\C^{0,1}_x})>0$,
\begin{eqnarray}
&& |(X,V)(t,s,x,v) - (X,V)(t',s',x',v')| \nonumber \\
&& \quad \quad \quad \quad \quad \quad \leq C (1+|v|)|(t,s,x,v) - (t',s',x',v')|, \nonumber 
\end{eqnarray} whenever $(t,s,x,v),(t',s',x',v') \in [0,T] \times \T^2 \times \R^2$, with $|v-v'|<1$.
\label{lemma:Gronwall}
\end{lemma}


\section{Construction of a reference trajectory}
\label{sec: referenceVNS}

The aim of this section is to construct a reference solution $(\ovl{f},\ovl{u}) $ of system (\ref{eq:VNS}), according to the return method, in such a way the characteristics associated to $\ovl{u}$, say $(\ovl{X},\ovl{V})$, verify the following property
\begin{align}
& \forall (x,v) \in \T^2 \times \R^2, \, \exists t\in \left[ \frac{T}{12}, \frac{11T}{12} \right] \textrm{ such that } \nonumber \\
&     \ovl{X}(t,0,x,v) \in H, \textrm{ with } |\ovl{V}(t,0,x,v) \cdot n_H| \geq 5.  \label{eq:referencecharacteristics}
\end{align} 

\begin{proposition}
Let $\omega \subset \T^2$ satisfy the strip assumption of Definition \ref{definition:strip}. There exists $T_0>0$ such that for any $T\geq T_0$, there exists a reference solution $(\ovl{f},\ovl{u})$ of system (\ref{eq:VNS}) such that 
\begin{align}
& \ovl{f} \in \C^{\infty}([0,T]\times \T^2; \S(\R^2)), \label{eq:regularityofthereferencedistribution} \\
& \ovl{u} \in \C^{\infty}([0,T]\times \T^2; \R^2), \label{eq:regularityofthereferencefield} \\
& (\ovl{f},\ovl{u})|_{t=0} = (0,0), \quad (\ovl{f},\ovl{u})|_{t=T} = (0,0), \label{eq:referencezeroatendandbeginning} \\
& supp(\ovl{f}) \subset (0,T) \times \omega \times \R^2,
\end{align} and such that the characteristics associated to $\ovl{u}$ satisfy (\ref{eq:referencecharacteristics}). 
\label{proposition:referencetrajectory}
\end{proposition}

\subsection{Global exact controllability of the Navier-Stokes system}

We recall a result due to Jean-Michel Coron and Andrei Fursikov, that guarantees the global exact controllability of the Navier-Stokes system on a surface without boundary (see \cite{CoronFursikov}). \par 
More precisely let $(M,g)$ be a connected, two-dimensional, orientable, compact, smooth Riemannian manifold without boundary. Let us denote by $T_x M$, as usual, the tangent space to $M$ at $x\in M$ and let $TM = \bigcup_{x\in M} T_x M$. For the definition of the differential operators $\div$, $\Delta$ and $\nabla \cdot$ on the manifold $M$ used below, we refer to \cite[Section 2]{CoronFursikov}. \par 
Let us choose and fix a particular solution of the Navier-Stokes system in $M$, i.e., let $\hat{y} \in \C^{\infty}\left( [0,\infty) \times M; TM \right)$ be such that
\begin{align}
& \hat{y}(t,x) \in T_x M, \quad \forall (t,x) \in [0,\infty) \times M, \\
& \div_x \hat{y}(t,x) = 0, \quad \forall (t,x) \in [0,\infty) \times M,
\end{align} and such that $\exists \hat{p} \in \C^{\infty}([0,\infty) \times M;\R)$ for which
\begin{equation*}
\partial_t \hat{y} - \Delta \hat{y} + \nabla_{\hat{y}}\cdot \hat{y} + \nabla \hat{p} = 0, \quad \textrm{in } (0,\infty) \times M.
\end{equation*} Then, we have the following controllability result.

\begin{thm}[Coron-Fursikov \cite{CoronFursikov}] Let $\tau > 0$, let $M_0 \subset M$ be an arbitrary open set and let $y_0 \in \C^{\infty}(M;TM)$ satisfying $y_0(x) \in T_x M$, for all $x\in M$, and that  $\div_x y_0 = 0$. Then, there exists a control $w \in \C^{\infty}(M \times [0,\infty); TM)$ verifying 
\begin{equation}
\supp w \subset (0,\tau) \times M_0
\label{eq:supportofthecontrol}
\end{equation} and such that the solution of the system
\begin{equation*}
\left\{  \begin{array}{ll}
\partial_t y - \Delta y + \nabla_y \cdot y + \nabla p = w, & (t,x) \in (0,\infty) \times M, \\
\div_x y(t,x) = 0, & (t,x) \in (0,\infty) \times M, \\
y_{|t=0} = y_0(x), & x \in M,
\end{array} \right.
\end{equation*} satisfies 
\begin{equation*}
y_{|t=\tau} = \hat{y}_{|t=\tau}.
\end{equation*}
\label{thm:CoronFursikov}
\end{thm}

This result guarantees that, given a fixed trajectory of the Navier-Stokes system, namely $\hat{y}$, given a time $\tau$, given an arbitrary initial state $y_0$ and given any open set $M_0 \subset M$, we can find a suitable force $w$, acting on $M_0$, allowing to pass from the initial state $y_0$ to the solution $\hat{y}$ in time $\tau$. We shall exploit Theorem \ref{thm:CoronFursikov} to construct a reference trajectory in the case of the torus.   

\subsection{Proof of Proposition \ref{proposition:referencetrajectory}}

Let us consider some $T_1,T_2,T_3,T_4$ with $T_3$ large enough, to be chosen later on, and
\begin{equation}
0 := T_0 < T_1 < T_2 < T_3 < T_4.
\label{eq:(0,T)}
\end{equation} We shall work separately on each interval $(T_i,T_{i+1})$, for $i=0,\dots, 3$. Let us fix from now on the values $0 < T_1 < T_2$.

\vspace{0.5em}
\textit{Step 1.} The reference solution in $[0,T_1]$. \par 
We set $(f_1,u_1) = (0,0)$, which trivially solves (\ref{eq:VNS}).

\vspace{0.5em}
\textit{Step 2.} The reference solution in $[T_1,T_2]$.\par 
At this step we use Theorem \ref{thm:CoronFursikov} to modify the reference trajectory in a convenient way. \par 
Indeed, since $\omega \subset \T^2$ satisfies the strip assumption, let us consider the constant vector field $\hat{y}(t,x) \equiv n_H$, where $n_H$ is given by Definition \ref{definition:strip}. Thus, $\hat{y}$ is a stationary solution of the Navier-Stokes system in $\T^2$. Let us apply Theorem \ref{thm:CoronFursikov} with
\begin{equation*}
\hat{y} = n_H, \quad M = \T^2, \quad M_0 = \omega, \quad \tau:= T_2 - T_1, \quad y_0 \equiv 0.
\end{equation*} This yields a control 
\begin{equation}
w_2 \in \C^{\infty}([0,T_2-T_1] \times \T^2;\R^2)
\label{eq:w2regularity}
\end{equation} with 
\begin{equation}
\supp w_2 \subset (0,T_2-T_1) \times \omega
\label{eq:supportw2}
\end{equation} and such that the corresponding solution to the Navier-Stokes system under this force, say $u_2$, satisfies
\begin{equation*}
u_2|_{t=T_2 - T_1} = n_H.
\end{equation*} To construct the associated distribution function, let us consider $\mathcal{Z}_1,\mathcal{Z}_2 \in \S(\R^2)$ such that
\begin{align}
\int_{\R^2} v_1 \mathcal{Z}_1 \ud v = 1, \quad \int_{\R^2} v_2 \mathcal{Z}_1 \ud v = 0, \quad \int_{\R^2} \mathcal{Z}_1 \ud v = 0, \label{eq:Z1integral} \\
\int_{\R^2} v_1 \mathcal{Z}_2 \ud v = 0, \quad \int_{\R^2} v_2 \mathcal{Z}_2 \ud v = 1, \quad \int_{\R^2} \mathcal{Z}_2 \ud v = 0. \label{eq:Z2integral}
\end{align} Then, define
\begin{equation}
f_2(t,x,v):= (\mathcal{Z}_1,\mathcal{Z}_2)(v)\cdot w_2(t,x), \quad \forall (t,x,v) \in (0,T_2 - T_1) \times \T^2 \times \R^2,
\end{equation} which gives
\begin{align}
f_2 \in \C^{\infty}([0,T_2-T_1]\times \T^2; \S(\R^2)), \label{eq:f2regularity} \\
j_{f_2}(t,x) = w_2(t,x), \quad \rho_{f_2}(t,x) = 0. \nonumber
\end{align} Consequently,
\begin{equation}
\partial_t u_2 + \left(u_2 \cdot \nabla \right)u_2 - \Delta u_2  + \nabla p_2 = j_{f_2} - \rho_{f_2}u_2, \quad \textrm{ in }(0,T_2-T_1) \times \T^2,
\label{eq:NS2}
\end{equation} and using (\ref{eq:supportw2}),
\begin{equation}
\partial_t f_2 + v\cdot \nabla_x f_2 + \div_v \left[ (u_2 - v) f_2  \right] = 0, \quad \textrm{in } (0,T_2-T_1) \times \left( \T^2 \setminus \omega \right) \times \R^2.
\label{eq:Vlasov2}
\end{equation}

\vspace{0.5em}
\textit{Step 3.} The reference solution in $[T_2,T_3]$. \par 
Let us choose $T_3>T_2$ large enough, to be chosen later on. During the interval $[T_2,T_3]$, we use the stationary solution $n_H$ to accelerate all the particles in the direction of $n_H$, as explained in detail in Step 5. \par

\vspace{0.5em}
\textit{Step 4.} The reference solution in $[T_3,T_4]$. \par 
Working as in Step 2, we use again Theorem \ref{thm:CoronFursikov} to steer the Navier-Stokes system from $n_H$ to $0$ in time $T_4-T_3$. This provides a control 
\begin{equation*}
w_4 \in \C^{\infty}([0,T_4-T_3]\times \T^2; \R^2)
\end{equation*} and
\begin{equation*}
\supp w_4 \subset (0,T_4-T_3) \times \omega,
\end{equation*} such that the corresponding solution of the Navier-Stokes system under the force $w_4$, say $u_4$, satisfies 
\begin{align}
u_4 \in \C^{\infty}([0,T_4-T_3] \times \T^2;\R^2), \nonumber \\
u_4|_{t=0} = n_H, \quad u_4|_{t= T_4 - T_3} = 0. \label{eq:u4finale}
\end{align} Thus, choosing $\mathcal{Z}_1$ and $\mathcal{Z}_2$ as in (\ref{eq:Z1integral}) and (\ref{eq:Z2integral}), we define
\begin{equation}
f_4(t,x,v):= (\mathcal{Z}_1,\mathcal{Z}_2)(v)\cdot w_4(t,x), \quad \forall (t,x,v) \in (0,T_4 - T_3) \times \T^2 \times \R^2.
\end{equation} By the same arguments as before, this yields
\begin{align}
\partial_t u_4 + \left(u_4 \cdot \nabla \right)u_4 - \Delta u_4  + \nabla p_4 = j_{f_4} - \rho_{f_4}u_4, \quad \textrm{in }(0,T_4-T_3) \times \T^2, \label{eq:NS4} \\
\partial_t f_4 + v\cdot \nabla_x f_4 + \div_v \left[ (u_4 - v) f_4  \right] = 0, \quad \textrm{in } (0,T_4-T_3) \times \left( \T^2 \setminus \omega \right) \times \R^2. \label{eq:Vlasov4}
\end{align} 

\vspace{0.5em}
\textit{Step 5. Conclusion.} \par 
Let us introduce the following parameters
\begin{equation}
\Lambda_0 := \max \left\{ \frac{d_0}{1-e^{-\frac{T_1}{2}}} , 5e^{T_1} \right\}, \quad d_0:= \max_{y\in \T^2} d(y,H),
\label{eq:choiceofLambda}
\end{equation} where $d_0$ is finite thanks to the compactness of $\T^2$ and the fact that $H$ is closed. Let us choose next $T_3> T_2$ large enough so that
\begin{align}
& \frac{1}{8}(T_3 -3T_2)^2 - T_3 \|u_2\|_{L^1(0,T_2 - T_1; L^{\infty}(\T^2))} \geq  \Lambda_0  + d_0, \label{eq:choiceofT3} \\
& T_3 \geq 3T_2 + 2 \left(  \Lambda_0  + \| u_2 \|_{L^1(0,T_2 - T_1; L^{\infty}(\T^2))} \right) + 10,  \label{eq:choiceofT3pourlesvitesses}
\end{align} where $u_2$ is defined as in Step 2. \par 
Next, according to (\ref{eq:(0,T)}), let us define the vector field
\begin{equation}
\ovl{u}(t) := \left\{  \begin{array}{ll}
0, & t\in [0,T_1], \\
u_2( t - T_1), & t\in [T_1,T_2], \\
n_H, & t\in [T_2,T_3], \\
u_4(t-T_3), & t\in [T_3,T_4],
\end{array} \right.
\label{eq:ubarre}
\end{equation} and the distribution function
\begin{equation*}
\ovl{f}(t) := \left\{  \begin{array}{ll}
0, & t\in [0,T_1], \\
f_2( t - T_1), & t\in [T_1,T_2], \\
0, & t\in [T_2,T_3], \\
f_4(t-T_3), & t\in [T_3,T_4].
\end{array} \right.
\end{equation*} Thanks to the previous definitions and using (\ref{eq:NS2}), (\ref{eq:Vlasov2}), (\ref{eq:NS4}) and (\ref{eq:Vlasov4}), we have, for $T = T_4$,  
\begin{align*}
\partial_t \ovl{u} + \left(\ovl{u} \cdot \nabla \right)\ovl{u} - \Delta \ovl{u}  + \nabla \ovl{p} = j_{\ovl{f}} - \rho_{\ovl{f}}\ovl{u}, \quad \textrm{in }(0,T) \times \T^2, \\
\partial_t \ovl{f} + v\cdot \nabla_x \ovl{f} + \div_v \left[ (\ovl{u} - v) \ovl{f}  \right] = 0, \quad \textrm{in } (0,T) \times \left( \T^2 \setminus \omega \right) \times \R^2. 
\end{align*} Let us now prove (\ref{eq:referencecharacteristics}). \par 
Let $(x,v) \in \T^2 \times \R^2$. We shall distinguish two cases.
\begin{description}
\item[Case 1. (High velocities)] Let us assume that $|v\cdot n_H | \geq \Lambda_0$. \par 
Thus, for any $s\in (\frac{T_1}{2},T_1)$, we have
\begin{align*}
| \left( \ovl{X}(s,0,x,v) - x \right) \cdot n_H | & = |(1-e^{-s})v\cdot n_H | \\
& \geq (1-e^{-\frac{T_1}{2}})|v\cdot n_H| > d_0,
\end{align*} thanks to the choice (\ref{eq:choiceofLambda}). Thus, thanks to the intermediate value theorem, there exists $t \in (0,T_1)$ such that $\ovl{X}(t,0,x,v) \in H$. Moreover,
\begin{equation*}
|\ovl{V}(t,0,x,v)\cdot n_H | = |e^{-t}v\cdot n_H | > e^{-T_1}\Lambda_0,
\end{equation*} which shows (\ref{eq:referencecharacteristics}) in this case.

\item[Case 2. (Low velocities)] Let us assume that $|v\cdot n_H| < \Lambda_0$. \par 
Taking $s \in (T_2, T_3)$, we can write, thanks to (\ref{eq:characteristics}) and (\ref{eq:ubarre}), 
\begin{equation}
\ovl{V}(s,0,x,v) = e^{-s}v + \int_0^{T_2-T_1} u_2(\tau, \ovl{X}(\tau)) \ud \tau + \int_{T_2}^s \ud \tau n_H,
\label{eq:bassesvitesses}
\end{equation} which, combined with (\ref{eq:choiceofT3}) entails 
\begin{align*}
& \left| (\ovl{X}(s,0,x,v) - x ) \cdot n_H \right| \\ 
& \quad \quad \quad = \left| (1-e^{-s}) v \cdot n_H + s \int_{0}^{T_2-T_1} u_2(\sigma, \ovl{X}(\sigma))\cdot n_H \ud \sigma  + \int_0^s \int_{T_2}^{\sigma}  \ud \tau \ud \sigma  \right| \\
& \quad \quad \quad = \left| (1-e^{-s}) v \cdot n_H + s \int_{0}^{T_2 - T_1} u_2(\sigma, \ovl{X}(\sigma)) \cdot n_H \ud \sigma  + \frac{(s - T_2)^2}{2}  \right| \\
& \quad \quad \quad \geq -|v\cdot n_H| - T_3 \|u_2 \|_{L^1_tL^{\infty}_x} + \frac{(s-T_2)^2}{2} \\
& \quad \quad \quad > d_0,
\end{align*} whenever $s > \frac{T_3 - T_2}{2}$. Consequently, by the intermediate value theorem, there exists $t \in (T_2,T_3)$ such that $\ovl{X}(t,0,x,v) \in H$. Moreover, the choice (\ref{eq:choiceofT3pourlesvitesses}) gives, through (\ref{eq:bassesvitesses}), that $|\ovl{V}(t,0,x,v)\cdot n_H | \geq 5 $, which entails (\ref{eq:referencecharacteristics}).
\end{description}


\section{Fixed-point argument}
\label{sec: fixedpointVNS}

Let $\epsilon\in(0,\epsilon_0)$ be fixed, with $\epsilon_0$ to be chosen later on. We shall define an operator $\V_{\epsilon}$ acting on a domain $\S_{\epsilon} \subset \C^0([0,T]\times \T^2 \times \R^2)$ to be defined below. The goal of this section is to show that $\V_{\epsilon}$ has a fixed point. \par 
Throughout all this section, we fix $f_0$ and $u_0$ as given in the statement of Theorem \ref{thm:controllability}.

\subsection{Definition of the operator}
In order to describe the set $\S_{\epsilon}$, let $\epsilon \in (0,\epsilon_0)$, to be precised later on, and $\gamma >2$. Then, set 
\begin{equation}
\delta_1:= \frac{\gamma}{2(\gamma+3)}, \quad \delta_2:= \frac{\gamma + 2}{\gamma+3}.
\label{eq:parametersdelta}
\end{equation} According to the notation of Section \ref{sec: notationVNS}, we define
\begin{align}
 \S_{\epsilon}:= & \left\{ g \in \C^{0,\delta_2} (Q_T);  \right. \nonumber \\
& \quad \quad \quad  \textbf{(a)}\,\,  \left\| \rho_g \right\|_{\C^{0,\delta_1}(\Omega_T)} \leq c_3 \epsilon, \nonumber \\
& \quad \quad \quad  \textbf{(b)} \,\, \|(1+|v|)^{\gamma+2} (\ovl{f} - g) \|_{L^{\infty}(Q_T)} \label{eq:domainoftheoperator} \\
& \quad \quad \quad  \quad \quad \quad  \quad \quad \leq c_1 \left( \|f_0\|_{\C^1(\T^2\times \R^2)} + \| (1+|v|)^{\gamma+2} f_0\|_{\C^0(\T^2 \times \R^2)} \right), \nonumber \\
& \quad \quad \quad  \textbf{(c)} \,\,\, \|(\ovl{f} - g) \|_{\C^{0,\delta_2}(Q_T)} \nonumber \\
& \quad \quad \quad  \quad \quad \quad  \quad \quad \leq c_2 \left( \|f_0\|_{\C^1(\T^2\times \R^2)} + \| (1+|v|)^{\gamma+2} f_0\|_{\C^0(\T^2 \times \R^2)} \right) \nonumber \Big\}, \nonumber 
\end{align} where $c_1,c_2,c_3$ are constants depending only on $T,\omega,\gamma, \delta_1$ and $\delta_2$ (see (\ref{eq:choiceofc1}), (\ref{eq:choiceofc2}) and (\ref{eq:choiceofc3}) for details) and $\ovl{f}$ is given by Proposition \ref{proposition:referencetrajectory}. We observe that, for $c_1,c_2,c_3$ large enough and $f_0\in \C^1(\T^2\times \R^2)$, with high moments in $v$, satisfying 
\begin{equation*}
\left\| \int_{\R^2} f_0(x,v) \ud v \right\|_{\C^{0,\delta_1}(\Omega_T)} \leq c_3\epsilon,
\end{equation*} we trivially have that $\ovl{f} + f_0 \in \S_{\epsilon}$. Thus, $\S_{\epsilon} \not = \emptyset$. \par

We define the operator $\V_{\epsilon}$ in three steps: 

\begin{enumerate}
\item First, we associate to each $g\in \S_{\epsilon}$ the solution of a suitable Navier-Stokes system, namely $u^g$. 
\item Secondly, we solve a Vlasov equation thanks to the field $u^g$, forcing the absorption of particles in $\omega$, which produces $\tilde{\V}_{\epsilon}[g]$. 
\item Thirdly, we perform a regular extension of $\tilde{\V}_{\epsilon}[g]$, which gives $\V_{\epsilon}[g].$    
\end{enumerate} We shall describe next these three steps in detail.


\subsection{Navier-Stokes system with a drag force interaction term}
Let $g\in \S_{\epsilon}$. The aim of this section is to give a sense to the associated Navier-Stokes system
\begin{equation}
\left\{ \begin{array}{ll}
\partial_t u^g + \left( u^g \cdot \nabla \right) u^g -\Delta_x u^g(t) + \nabla_x p^g(t) = j_g(t) - \rho_g(t) u^g  , & \textrm{in }\Omega_T, \\
\div_x u^g(t,x) = 0, & \textrm{in } \Omega_T, \\
u^g_{|t=0} = u_0, & \textrm{in }\T^2,
\end{array} \right.
\label{eq:NavierStokesforg}
\end{equation} where
\begin{equation*}
j_g(t,x):= \int_{\R^2} v g(t,x,v) \ud v, \quad \rho_g(t,x) := \int_{\R^2} g(t,x,v)\ud v,
\end{equation*} and $u_0$ satisfies (\ref{eq:massechampinitiale}). Let us observe that the interaction between the fluid and the distribution function is taken into account through the term $j_g - \rho_g u$.

\begin{definition}
A time-dependent vector field $u$ is a weak solution of (\ref{eq:NavierStokesforg}) whenever
\begin{equation}
u \in \C^0([0,T];\V_{\sigma}') \cap L^{\infty}(0,T;\Hdiv) \cap L^2(0,T;\V_{\sigma}),
\label{eq:regularityofNSg}
\end{equation} and for any $\psi \in \C^1([0,T];\V_{\sigma})$ and $t\in (0,T]$, one has 
\begin{align}
& \int_{\T^2} u(t)\psi(t) \ud x + \int_0^t \int_{\T^2} \left( \nabla u : \nabla \psi - u \otimes u \cdot \nabla \psi -u\partial_t \psi \right) \ud s \ud x \nonumber \\
& \quad \quad \quad \quad = \int_{\T^2} u_0 \psi(0) \ud x + \int_0^t \int_{\T^2} \left( j_g(s) - \rho_g(s) u(s) \right) \psi(s)  \ud s \ud x. \label{eq:formulationfaibleNSg}
\end{align} 
\label{definition:NSsolution}
\end{definition}

\begin{proposition}
There exists $\epsilon_0 >0$ small enough such that for any $\epsilon \leq \epsilon_0$, $g\in \S_{\epsilon}$ and initial data $f_0$ and $u_0$ satisfying (\ref{eq:smalldata}) and (\ref{eq:massechampinitiale}), there exists a unique weak solution of system (\ref{eq:NavierStokesforg}) in the sense of Definition \ref{definition:NSsolution}, Moreover, this solution satisfies, for any $t\in [0,T]$,
\begin{align}
& \|u^g(t)\|^2_{L^2(\T^2)} + \int_0^t \|\nabla u^g(s)\|^2_{L^2(\T^2)} \ud s  \label{eq:estimationuniformeNSg} \\
& \quad \quad \quad \quad \quad \quad \quad \quad \leq 2 e^T\left( M^2 + T(1 + \| j_{\ovl{f}} \|^2_{L^{\infty}(0,T;L^2(\T^2))}) \right),  \nonumber
\end{align}  where $M>0$ is given by (\ref{eq:massechampinitiale}).
\label{proposition:existenceNSg}
\end{proposition}

Let us show a property of $j_g$ that will be important in the proof of the result above.

\begin{lemma}
Let $\epsilon>0$. For any $g \in \S_{\epsilon}$, we have
\begin{equation}
\sup_{t\in [0,T]} \|j_g(t)\|^2_{L^2(\T^2)} \leq 2 \left( \mathcal{I}^2c_1 \epsilon^2 + \| j_{\ovl{f}} \|_{L^{\infty}(0,T;L^2(\T^2))}^2 \right),
\label{eq:jg}
\end{equation} where
\begin{equation}
\mathcal{I}:= \int_{\R^2} \frac{ |v| \ud v}{(1+|v|)^{\gamma +2}} < \infty.  
\label{eq:constantI}
\end{equation}
\label{lemma:jg}
\end{lemma}

\begin{proof}
We write, by the triangular inequality,
\begin{align}
\|j_g(t)\|_{L^2(\T^2)^2}^2 & = \int_{\T^2} \left| \int_{\R^2} v g(t,x,v) \ud v \right|^2 \ud x \nonumber \\
& = \int_{\T^2} \left| \int_{\R^2} v \left( g - \ovl{f} + \ovl{f} \right)(t,x,v) \ud v \right|^2 \ud x \nonumber \\
& \leq \int_{\T^2} \left( \left| \int_{\R^2} v(g-\ovl{f}) \ud v \right| + \left| \int_{\R^2} v\ovl{f} \ud v \right| \right)^2 \ud x \nonumber \\
& \leq 2 \int_{\T^2} \left( \int_{\R^2} |v||g-\ovl{f}| \ud v  \right)^2 \ud x + 2 \int_{\T^2} \left| \int_{\R^2} v\ovl{f} \ud v  \right|^2 \ud x.  \label{eq:lemma411}
\end{align} Let us note that, from (\ref{eq:regularityofthereferencedistribution}), we have 
\begin{equation}
\| j_{\ovl{f}} \|_{L^{\infty}(0,T;L^2(\T^2))}^2 := \sup_{t\in [0,T]} \int_{\T^2} \left| \int_{\R^2} v\ovl{f}(t) \ud v  \right|^2 \ud x < \infty,
\label{eq:lemma413}
\end{equation} which is a positive constant, independent from $g$. \par 
We have to treat the first part of (\ref{eq:lemma411}). Indeed,
\begin{align}
& \int_{\T^2} \left(  \int_{\R^2} |v||(g-\ovl{f})(t,x,v)| \ud v\right)^2 \ud x \nonumber \\
& \quad \quad  \leq   \left( \int_{\R^2} \frac{|v| \ud v}{(1+|v|)^{\gamma + 2}} \right)^2 \| (1+|v|)^{\gamma+2}f_0 \|_{L^{\infty}}^2, \nonumber \\
& \quad \quad \leq  \mathcal{I}^2 c_1^2 \epsilon^2, \label{eq:lemma412}
\end{align} where we have used (\ref{eq:smalldata}), point (b) and (\ref{eq:constantI}).
Finally, putting together (\ref{eq:lemma412}), (\ref{eq:lemma413}) and (\ref{eq:lemma411}), we obtain (\ref{eq:jg}).
\end{proof}

\begin{proof}[Proof of Proposition \ref{proposition:existenceNSg}]
Firstly, we construct a solution, which proves the existence part. Secondly, we show that this solution must be unique.\par 
\vspace{0.5em}
\textit{1. Existence.} Let us consider the following iterative scheme, for every $n\in \N$,
\begin{equation*}
\left\{   \begin{array}{ll}
\partial_t u^{n+1} + \left( u^{n+1} \cdot \nabla \right)u^{n+1} - \Delta u^{n+1} + \nabla p^{n+1} = j_g - \rho_g u^n, & (0,T) \times \T^2,\\ 
\div u^{n+1} = 0, & t \in (0,T), \\
u^{n+1}|_{t=0} = u_0, & x \in \T^2.
\end{array} \right.
\end{equation*} We observe that, since $g \in \S_{\varepsilon}$, we have $j_g \in L^2(0,T;\V_{\sigma}')$. In addition, since $u_0 \in \V_{\sigma}$ by (\ref{eq:massechampinitiale}), Theorem \ref{thm:Leray} yields
\begin{equation*}
u^1 \in L^{\infty}(0,T;\Hdiv) \cap L^2(0,T;\V_{\sigma}). 
\end{equation*} Thus, by induction,
\begin{equation*}
u^n \in L^{\infty}(0,T;\Hdiv) \cap L^2(0,T;\V_{\sigma}), \quad \forall n \in \N. 
\end{equation*} Furthermore, according to Definition \ref{definition:Leraysolution}, we deduce that for any $\psi \in \C^1(\R^+;\V_{\sigma})$, and any $t\in (0,T]$ and $n\in \N$,
\begin{align}
& \int_{\T^2} u^{n+1}(t)\psi(t) \ud x + \int_0^t \int_{\T^2} \left( \nabla u^{n+1} : \nabla \psi - u^{n+1} \otimes u^{n+1} \cdot \nabla \psi -u^{n+1} \partial_t \psi \right) \ud s \ud x \nonumber \\
& \quad \quad \quad \quad = \int_{\T^2} u_0 \psi(0) \ud x + \int_0^t \int_{\T^2} \left( j_g(s) - \rho_g(s) u^n(s) \right) \psi(s)  \ud s \ud x. \label{eq:faibleNSenn}
\end{align} Moreover, the energy estimate (\ref{eq:NSestimate}) gives, for any $n\in \N$,
\begin{align}
& \|u^{n+1}(t)\|_{L^2}^2 + \int_0^t \| \nabla u^{n+1} (s)\|^2_{L^2} \ud s \nonumber \\
& \quad \quad \quad \leq e^t \left( \|u_0\|_{L^2}^2 + \int_0^t \left( \|j_g(s)\|_{L^2}^2 + \|\rho_g u^n(s)\|_{L^2}^2 \right)\ud s   \right). 
\label{eq:estimateinn}
\end{align} Our aim is to obtain uniform estimates with respect to $n\in \N$. \par 

Indeed, choosing 
\begin{equation}
\epsilon_0 \leq \min \left\{ 1, \frac{1}{\mathcal{I}c_1}, \frac{1}{c_3\sqrt{T}}  \frac{M}{c_3 \sqrt{2Te^T[M^2 + T(1 + \| j_{\ovl{f}}\|_{L^{\infty}_tL^2_x}^2)]}}  \right\},
\label{eq:epsilon0}
\end{equation} where $M>0$ is given by (\ref{eq:massechampinitiale}) and $\mathcal{I}$ is given by (\ref{eq:constantI}), one has 
\begin{align}
& \| u^n (t) \|_{L^2}^2 + \int_0^t \|\nabla u^n(s)\|_{L^2}^2 \ud s  \label{eq:claimuniformestimate} \\ 
& \quad \quad \quad \quad \quad \quad \quad \leq 2e^T \left[ M^2 + T( 1 + \|j_{\ovl{f}}\|_{L^{\infty}_t L^2_x}^2) \right], \quad \forall n\in \N. \nonumber
\end{align} To prove this claim, let us proceed by induction. Indeed, for $n=0$, (\ref{eq:estimateinn}) yields, thanks to point (a), (\ref{eq:massechampinitiale}), (\ref{eq:smalldata}) and Lemma \ref{lemma:jg},
\begin{align}
& \|u^1(t) \|_{L^2}^2 + \int_0^t \| \nabla u^1(s) \|_{L^2}^2 \ud s \nonumber \\
& \quad \quad \quad \leq e^T \left( \|u_0\|_{L^2}^2 + \int_0^t \|j_g(s)\|_{L^2}^2 \ud s + T\|\rho_g\|^2_{L^{\infty}}\|u_0\|_{L^2}^2  \right) \nonumber \\
& \quad \quad \quad \leq e^T\left( M^2 + 2T(1 + \| j_{\ovl{f}}\|_{L^{\infty}_t L^2_x}^2 ) + Tc_3^2\epsilon^2 M^2  \right) \nonumber \\
& \quad \quad \quad \leq 2e^T \left[ M^2 + T( 1 + \|j_{\ovl{f}}\|^2_{L^{\infty}_t L^2_x}) \right], \nonumber
\end{align} using the choice (\ref{eq:epsilon0}). \par 
Let now be any $N\in \N^*$ and suppose that (\ref{eq:claimuniformestimate}) holds for any $n\in \N$ up to $N-1$. Thus, in the same fashion as before, (\ref{eq:estimateinn}) yields 
\begin{align}
& \|u^N(t) \|_{L^2}^2 + \int_0^t \| \nabla u^N(s) \|_{L^2}^2 \ud s \nonumber \\
& \quad \quad \quad \leq e^T \left( \|u_0\|_{L^2}^2 + \int_0^t \|j_g(s)\|_{L^2}^2 \ud s + T\|\rho_g\|^2_{L^{\infty}}\|u^{N-1}\|_{L^{\infty}_t L^2_x}^2  \right) \nonumber \\
& \quad \quad \quad \leq e^T\left( M^2 + 2T( 1 + \| j_{\ovl{f}} \|^2_{L^{\infty}_t L^2_x} ) + T c_3^2 \epsilon^2 \|u^{N-1}\|_{L^{\infty}_t L^2_x}^2   \right) \nonumber \\
& \quad \quad \quad \leq 2 e^T \left[ M^2 + T( 1 + \| j_{\ovl{f}} \|^2_{L^{\infty}_t L^2_x}) \right]. \nonumber
\end{align} This shows (\ref{eq:claimuniformestimate}).

Consequently, $(u^n)_{n\in \N}$ is uniformly bounded in $L^2(0,T;\V_{\sigma})$, which implies that $\exists u \in L^2(0,T;\V_{\sigma})$ such that
\begin{align}
& u^n \rightharpoonup u \textrm{ in } L^2(0,T;\V_{\sigma}), \label{eq:convergencefaible} \\
& u^n \rightarrow u \textrm{ in } L^2(0,T;\Hdiv), \label{eq:convergenceforte}
\end{align} thanks to Banach-Alaoglu's theorem and Rellich's theorem. Thus, a compactness argument allows to pass to the limit in (\ref{eq:faibleNSenn}), which gives (\ref{eq:formulationfaibleNSg}). This can be done in detail following \cite[Section 2.2.4]{Chemin}. \par 
Moreover, (\ref{eq:convergencefaible}) and (\ref{eq:convergenceforte}), combined with (\ref{eq:claimuniformestimate}), gives (\ref{eq:estimationuniformeNSg}). \par 
\vspace{+0.5em}


\textit{2. Uniqueness.} Let us prove next that the solution constructed above is unique. Consider another solution of (\ref{eq:NavierStokesforg}), namely $v$. Thus, by Theorem \ref{thm:Leraystabilite}, $u-v$ satisfies the estimate
\begin{align*}
& \|(u-v)(t)\|_{L^2}^2 + \int_0^t \|\nabla (u-v)(s)\|_{L^2}^2 \ud s \\
& \quad \quad \quad \leq exp\left( T + cE(t)^2  \right) \int_0^t \|\rho_g(u-v)(s)\|_{L^2}^2 \ud s,   
\end{align*} for some constant $c>0$ and
\begin{equation}
E(t) := e^T\|u_0\|^2_{L^2} + e^T\min\left\{ \int_0^T \|j_g - \rho_g u \|^2_{L^2} \ud s, \int_0^T \|j_g - \rho_g v \|^2_{L^2} \ud s        \right\}. 
\end{equation} We shall prove that $E(t)$ can be bounded independently from $u$ or $v$. Indeed, since $v$ is a solution of (\ref{eq:NavierStokesforg}), the estimate (\ref{eq:NSestimate}) gives 
\begin{equation*}
\| v(t) \|_{L^2}^2 \leq e^T \left( \|u_0\|_{L^2}^2 + \int_0^t \|j_g(s)\|_{L^2}^2 \ud s + \|\rho_g\|_{L^{\infty}}^2 \int_0^t \|v(s)\|_{L^2}^2 \ud s    \right). 
\end{equation*} Thus, (\ref{eq:smalldata}), Lemma \ref{lemma:jg} and point (a) combined with Gronwall's lemma give 
\begin{equation} 
\sup_{t\in [0,T]} \|v(t)\|_{L^2}^2 \leq \tilde{C}(T,\gamma,\ovl{f}),
\end{equation} for some constant $\tilde{C}>0$. This, using (\ref{eq:estimationuniformeNSg}), yields
\begin{equation}
\sup_{t\in [0,T]} E(t) \leq C,
\end{equation} for some constant $C>0$. Then, we find, by point (a),
\begin{equation*}
\|(u-v)(t)\|_{L^2}^2 + \int_0^t \|\nabla (u-v)(s)\|_{L^2}^2 \ud s \lesssim \int_0^t \|(u-v)(s)\|_{L^2}^2 \ud s,   
\end{equation*} for any $t\in [0,T]$, which, thanks to Gronwall's lemma allows to write
\begin{equation*}
\sup_{t\in [0,T]} \|(u-v)(t)\|_{L^2} \leq 0.
\end{equation*} Henceforth, $u\equiv v$.
\end{proof}


We now provide further regularity properties of $u^g$ that will be important to define the characteristics associated to $-v + u^g$, used in the sequel.

\begin{proposition}
Let $\epsilon \leq \epsilon_0$, where $\epsilon_0$ is given by Proposition \ref{proposition:existenceNSg}. Then, there exists a constant $K_1 = K_1(T,M,\ovl{f})>0$, such that for any $g\in \S_{\epsilon}$, the solution of (\ref{eq:NavierStokesforg}) satisfies 
\begin{equation}
\|u^g\|_{L^2(0,T;L^{\infty}(\T^2))} \leq K_1.
\label{eq:champuniforme}
\end{equation} Moreover,
\begin{equation}
u^g \in L^2(0,T;\C^1(\T^2;\R^2)) \cap \C^0([0,T];\V_{\sigma}) \cap L^2(0,T;H^2(\T^2)).
\label{eq:champLipschitz}
\end{equation}
\end{proposition}

\begin{proof}
We shall prove first (\ref{eq:champuniforme}), by using a regularity result for the Navier-Stokes system (Theorem \ref{thm:Lerayregularite}). Secondly, we prove (\ref{eq:champLipschitz}) thanks to the regularising properties of the Stokes system (Theorem \ref{thm:Stokesclassical}).\par 
\vspace{0.5em}
\textit{1. Estimate} $L^2_t L^{\infty}_x$.  Let us consider, for any fixed $g\in \S_{\epsilon}$, the solution of (\ref{eq:NavierStokesforg}) given by Proposition \ref{proposition:existenceNSg}, that we note $u$. We observe that Theorem \ref{thm:Lerayregularite} is stated under the mean-free assumption, which requires to take care of the mean of $u$ in our case. Let us set
\begin{equation}
\hat{u}(t,x) := u(t,x) - m_u(t), \quad (t,x) \in [0,T]\times \T^2,
\end{equation} with
\begin{equation}
m_u(t) := \int_{\T^2} u(t,x) \ud x.
\end{equation} We deduce from (\ref{eq:regularityofNSg}) that, according to (\ref{eq:FouriercoefficientsVNS}),
\begin{equation}
\hat{u}\in L^{\infty}(0,T;\Hdiv \cap L^2_0) \cap L^2(0,T;\V_{\sigma}\cap H^1_0). 
\end{equation} Moreover, $\hat{u}$ satisfies the system
\begin{equation}
\left\{  \begin{array}{ll}
\partial_t \hat{u} + \left(  \hat{u}\cdot \nabla \right) \hat{u} -\Delta \hat{u} + \nabla p = F_{hom}(t,x), & (t,x)\in (0,T)\times \T^2, \\
\div \hat{u} = 0, & (t,x) \in (0,T)\times \T^2 , \\
\int_{\T^2} \hat{u}(t,x) \ud x = 0, & t\in (0,T), \\
\hat{u}_{|t=0} = u_0 - \int_{\T^2} u_0 \ud x, & x\in \T^2, 
\end{array}  \right.
\label{eq:systemeamoyennenulle}
\end{equation} with
\begin{equation}
F_{hom}(t,x) := j_g - \int_{\T^2}j_f \ud x + \rho_g u -\int_{\T^2} \rho_g u \ud x + \left( m_u(t)\cdot \nabla \right) u.
\end{equation} Our goal is to apply Theorem \ref{thm:Lerayregularite} to system (\ref{eq:systemeamoyennenulle}). In what follows, we shall use (\ref{eq:normsFourier}) systematically.\par 
Firstly, let us observe that
\begin{equation*}
\|\hat{u}_{|t=0}\|_{H_0^{\frac{1}{2}}(\T^2)}^2 \leq \|u_0\|_{H^{\frac{1}{2}}(\T^2)}^2 < M^2, 
\end{equation*} according to (\ref{eq:massechampinitiale}). \par 
In order to treat the source term, we write
\begin{equation}
F_{hom}:= \mathcal{T}_1 + \mathcal{T}_2 + \mathcal{T}_3,
\label{eq:T1T2T3}
\end{equation} with
\begin{equation*}
\mathcal{T}_1 := j_g - \int_{\T^2} j_g \ud x, \quad \mathcal{T}_2 := \rho_g u - \int_{\T^2} \rho_g u \ud x, \quad \mathcal{T}_3 := \left( m_u(t)\cdot \nabla \right) u. 
\end{equation*} For the first term, Lemma \ref{lemma:jg} allows to write
\begin{align}
\|\mathcal{T}_1\|_{L^{\infty}(0,T; H^{-\frac{1}{2}}_0)}^2 &= \sup_{t\in[0,T]} \sum_{k\in\Z^2\setminus\left\{ 0 \right\}} \frac{1}{|k|} |j_{g,k}(t)|^2 \label{eq:majorationT1} \\
& \leq \sup_{t\in[0,T]} \sum_{k\in\Z^2}|j_{g,k}(t)|^2 \nonumber \\
& = \|j_g\|_{L^{\infty}(0,T;L^2)}^2 < 2\left( 1 + \|j_{\ovl{f}}\|_{L^{\infty}_t L^2_x}^2 \right). \nonumber
\end{align} Analogously, using (\ref{eq:smalldata}), point (a) and (\ref{eq:estimationuniformeNSg}), we obtain
\begin{align}
\|\mathcal{T}_2\|_{L^{\infty}(0,T; H^{-\frac{1}{2}}_0)} & = \sup_{t\in[0,T]} \sum_{k\in\Z^2\setminus\left\{ 0 \right\}} \frac{1}{|k|} |\left( \rho_g u \right)_k(t)|^2 \label{eq:majorationT2} \\
& \leq \| \rho_g\|^2_{L^{\infty}(Q_T)}\sup_{t\in[0,T]} \sum_{k\in\Z^2}|u_k(t)|^2 \leq c_3^2 \epsilon^2 \|u\|_{L^{\infty}(0,T;L^2)}^2 \nonumber \\
& \leq c_3^2 \epsilon^2 2 e^T \left( M^2 + T(1 + \|j_{\ovl{f}}\|_{L^{\infty}_t L^2_x}^2)   \right). \nonumber
\end{align} Finally, we show that $\mathcal{T}_3 \in L^2(0,T;L^2_0)$ which, a fortiori, implies $\mathcal{T}_3 \in L^2(0,T;H_0^{-\frac{1}{2}})$. Indeed, for any $t\in [0,T]$,
\begin{align}
\|m_u(t)\cdot \nabla u \|_{L^2(\T^2)}^2 & = \int_{\T^2} \sum_{j=1,2} \left| \sum_{i=1,2} m_u^i(t) \partial_i u^j(t)  \right|^2 \ud x \nonumber \\
& \leq \sup_{t\in [0,T]}|m_u(t)|^2 \sum_{i,j=1,2} \int_{\T^2} \left| \partial_i u^j(t)   \right|^2 \ud x. \nonumber 
\end{align} On the other hand, by Jensen's inequality and (\ref{eq:estimationuniformeNSg}), we get, for any $t\in [0,T]$,
\begin{align}
|m_u(t)|^2 & \leq \int_{\T^2} |u(t,x)|^2 \ud x \label{eq:meanuniforme} \\
& = \|u(t)\|_{L^2(\T^2)}^2 \leq 2e^T\left( M^2 + T(1 + \|j_{\ovl{f}}\|^2_{L^{\infty}_t L^2_x} ) \right). \nonumber
\end{align} Thus, we obtain
\begin{align}
& \int_0^T \| m_u(t)\cdot \nabla u\|_{L^2(\T^2)}^2 \ud t \label{eq:convectionmoyenne} \\
& \quad \quad \quad \quad \lesssim 2e^T\left( M^2 + T(1 + \|j_{\ovl{f}}\|_{L^{\infty}_t L^2_x}^2 \right) \int_0^T \|\nabla u (t) \|_{L^2(\T^2)}^2 \ud t \nonumber \\
& \quad \quad \quad \quad \lesssim 4e^{2T} \left( M^2 + T(1 + \|j_{\ovl{f}}\|_{L^{\infty}_t L^2_x}^2) \right)^2, \nonumber 
\end{align} thanks again to (\ref{eq:estimationuniformeNSg}). \par 
Hence, from (\ref{eq:T1T2T3}), we deduce that $F_{hom} \in L^2(0,T;H_0^{-\frac{1}{2}})$ and then, Theorem \ref{thm:Lerayregularite} entails
\begin{equation}
\hat{u} \in \C^0([0,T];H_0^{\frac{1}{2}}) \cap L^2(0,T;H_0^{\frac{3}{2}})
\end{equation} and gives that, for every $t\in [0,T]$,
\begin{align}
& \|\hat{u}(t)\|_{H^{\frac{1}{2}}_0}^2 + \int_0^t \| \nabla \hat{u}(s)\|_{H_0^{\frac{1}{2}}}^2 \ud s \nonumber \\
& \quad \quad \quad \leq exp\left( C\int_0^t \|\nabla \hat{u} \|_{L^2}^2 \ud s   \right) \left( \|\hat{u}_0 \|_{H_0^{\frac{1}{2}}}^2 + \int_0^t \|F_{hom}(s) \|^2_{H_0^{-\frac{1}{2}}} \ud s\right). \label{eq:estimateH32pourconclure}
\end{align} In addition, combining (\ref{eq:T1T2T3}), (\ref{eq:majorationT1}) and (\ref{eq:majorationT2}) and (\ref{eq:convectionmoyenne}), yields
\begin{align}
& \int_0^t \|F_{hom}(s) \|^2_{H_0^{-\frac{1}{2}}} \ud s  \nonumber \\
& \quad \quad \leq T \left( \| \mathcal{T}_1 + \mathcal{T}_2\|_{L^{\infty}(0,T;H^{-\frac{1}{2}})} \right) + \int_0^t \| \left( m_u(t) \cdot \nabla \right) u\|_{L^2(\T^2)}^2 \nonumber \\
& \quad \quad \leq C(T,M,\ovl{f}). \nonumber
\end{align} Consequently, injecting this in (\ref{eq:estimateH32pourconclure}) and using (\ref{eq:estimationuniformeNSg}), we find
\begin{equation}
\|\hat{u}(t)\|_{H^{\frac{1}{2}}_0}^2 + \int_0^t \| \nabla \hat{u}(s)\|_{H_0^{\frac{1}{2}}}^2 \ud s \leq e^{CT} C(T,M,\ovl{f}). \label{eq:normeH12bornee} 
\end{equation} Finally, using the injection $H^{\frac{3}{2}}(\T^2) \hookrightarrow L^{\infty}(\T^2)$, we deduce
\begin{align}
& \int_0^t \|u(s)\|_{L^{\infty}(\T^2)}^2 \ud s \leq C_S \int_0^t \|u(s)\|_{H^{\frac{3}{2}}(\T^2)}^2 \ud s \nonumber \\
& \quad \quad \quad \quad  \lesssim \int_0^t \sum_{k\in \Z^2} \left( 1 + |k|^2 \right)^{\frac{3}{2}} |u_k(s)|^2 \ud s \nonumber \\
& \quad \quad \quad \quad  \lesssim \int_0^t |m_u(s)|^2 \ud s + \int_0^t \sum_{k\in \Z^2 \setminus \left\{ 0 \right\}} |k|^3 |\hat{u}_k(s)|^2 \ud s  \nonumber \\
& \quad \quad \quad \quad  \lesssim T \|u\|_{L^{\infty}(0,T;L^2)} + \int_0^T \|\nabla \hat{u}(s)\|_{H_0^{\frac{1}{2}}(\T^2)}^2 \ud s \nonumber \\
& \quad \quad \quad \quad  \lesssim C(T,M,\ovl{f}), \nonumber
\end{align} thanks to (\ref{eq:meanuniforme}) and (\ref{eq:normeH12bornee}). This gives (\ref{eq:champuniforme}) with $K_1 \geq C(T,M,\ovl{f})$.\par

\vspace{0.5em}
\textit{2. $L^2_t \C^1_x$ regularity. }Let us show next (\ref{eq:champLipschitz}), thanks to the regularity properties of the Stokes system. Indeed, we may rewrite (\ref{eq:NavierStokesforg}) as
\begin{equation}
\left\{ \begin{array}{ll}
\partial_t u - \Delta u + \nabla p = F_{source}(t,x), & (t,x) \in \Omega_T, \\
\div u(t,x) = 0, & (t,x) \in \Omega_T, \\
u|_{t=0}= u_0, & x \in \T^2,
\end{array} \right.
\label{eq:Stokesbootstrap}
\end{equation} with 
\begin{equation}
F_{source}:= j_g - \rho_g u - \left( u \cdot \nabla \right) u.
\end{equation} According to the previous discussion, $j_g - \rho_gu \in L^2(0,T;L^2(\T^2))$. Consequently, we have to estimate the convection term $u\cdot \nabla u$. In order to to do so, we use the following argument. \par 

Let $r\in \N$ with $r\geq 2$. Then, by H\"older's inequality, and the Cauchy-Schwarz's inequality, we have
\begin{align}
\| \left( u(t)\cdot \nabla \right)u(t) \|_{L^r(\T^2)}^r  & \lesssim \int_{\T^2} \sum_{i,j=1,2} |u^i(t) \partial_i u^j(t) |^r \ud x \nonumber \\
& \lesssim \int_{\T^2} \left(  \sum_{i=1,2} |u^i(t)|^{2r} \right)^{\frac{1}{2}} \left(  \sum_{i,j=1,2} |\partial_i u^j(t)|^{2r} \right)^{\frac{1}{2}} \ud x \label{eq:convectionLr} \\ 
& \lesssim \left( \int_{\T^2} \sum_{i=1,2} |u^i(t)|^{2r} \ud x  \right)^{\frac{1}{2}} \left( \int_{\T^2}  \sum_{i,j=1,2} |\partial_i u^j(t)|^{2r}   \ud x  \right)^{\frac{1}{2}} \nonumber \\
& \lesssim \|u(t)\|_{L^{2r}(\T^2)}^r \| \nabla u(t) \|_{L^{2r}(\T^2)}^r, \nonumber
\end{align} for any $t\in [0,T]$. Then, choosing $r=2$ in the estimate above, the injection $H^{\frac{1}{2}}(\T^2) \hookrightarrow L^4(\T^2)$ (see \cite[p.81]{Chemin}), allows to deduce
\begin{align}
& \int_0^T \| \left( u(t)\cdot \nabla \right)u(t) \|_{L^2(\T^2)}^2 \ud t \nonumber \\
& \quad \quad \quad  \lesssim \int_0^T \|u(t)\|_{H^{\frac{1}{2}}(\T^2)}^2 \| \nabla u(t) \|_{H^{\frac{1}{2}}(\T^2)}^2 \ud t \nonumber \\
& \quad \quad \quad  \lesssim \sup_{t\in [0,T]} \|u(t)\|_{H^{\frac{1}{2}}(\T^2)}^2 \int_0^T \| \nabla u(t) \|_{H^{\frac{1}{2}}(\T^2)}^2 \ud t < \infty, \nonumber 
\end{align} thanks to (\ref{eq:normeH12bornee}). Thus, $F_{source} \in L^2(0,T;L^2(\T^2))$. \par 
Consequently, as $u_0$ is regular enough by (\ref{eq:massechampinitiale}), Theorem \ref{thm:Stokesclassical} with the choice $s=q=2$, yields
\begin{equation*}
u \in L^2(0,T;H^2(\T^2) \cap \V_{\sigma}) \quad \textrm{and} \quad \partial_t u \in L^2(0,T;\mathcal{H}). 
\end{equation*} This allows to deduce (\cite[Theorem II.5.13, p.101]{Boyer}) that 
\begin{equation}
u\in \C^0([0,T];\V_{\sigma}).
\label{eq:C0H1}
\end{equation} Let us perform next a bootstrap argument. Let us choose $r=3$ in (\ref{eq:convectionLr}), which allows to deduce
\begin{align*}
\int_0^T \| u\cdot \nabla u (t) \|_{L^3(\T^2)}^3 \ud t  & \lesssim \int_0^T \|u(t)\|^3_{L^6(\T^2)} \|\nabla u(t)\|^3_{L^6(\T^2)} \ud t \\
& \lesssim \| u \|^3_{L_t^{\infty} L^6_x} \int_0^T \| \nabla u(t)\|^3_{L^6(\T^2)} \ud t \\
& \lesssim \|u\|_{L_t^{\infty}H^1_x}^3 \|u\|_{L^2_t H^2_x}^3,
\end{align*} as $H^1(\T^2) \hookrightarrow L^p(\T^2)$ for any $p\geq 2$ (\cite[Th. 5.6.6, p. 270]{Evans}).   Thus, we deduce that $u \cdot \nabla  u \in L^2(0,T;L^3(\T^2))$. \par 
In addition, using points (c) and (a), we have
\begin{align*}
& j_g \in L^2(0,T;L^p(\T^2)), \quad \forall p \geq 2, \\
& \rho_g \in L^{\infty}(\Omega_T),  
\end{align*} which entails that
\begin{equation*}
j_g - \rho_g u \in L^2(0,T;L^3(\T^2)),
\end{equation*} as $u \in L^2(0,T;L^3(\T^2))$, thanks to (\ref{eq:C0H1}) and the Sobolev embedding. Then, applying Theorem \ref{thm:Stokesclassical} to system (\ref{eq:Stokesbootstrap}) with $s = 2, q=3$, as $u_0$ is regular enough, thanks to (\ref{eq:massechampinitiale}), we deduce that
\begin{equation*}
u \in L^2(0,T;W^{2,3}(\T^2)) \textrm{ and } \quad \partial_t u \in L^2(0,T;L^3(\T^2)).
\end{equation*} Finally, the injection $W^{2,3}(\T^2) \hookrightarrow \C^1(\T^2)$ (\cite[Th. 5.6.6, p. 270]{Evans}) gives (\ref{eq:champLipschitz}).
\end{proof}


\subsection{Absorption}
In order to describe the absorption procedure, we have to introduce some definitions (see \cite[p. 369]{Glass}). According to Definition \ref{definition:strip}, there exists $\delta_0>0$ such that $H_{2\delta_0} \subset \omega$. Let us choose
\begin{equation}
\delta := \min \left\{ \delta_0, \frac{1}{2}, e^{\frac{T}{200}} - 1, \frac{1}{4 K_1^2}  \right\},
\label{eq:choiceofdelta}   
\end{equation} where $K_1$ is given by (\ref{eq:champuniforme}). The choice of this parameter will be useful in Section \ref{sec: fixedpointrelevantVNS}. According to this choice of $\delta$, we set
\begin{eqnarray}
&& \gamma^{-}:= \left\{ (x,v) \in \partial H_{\delta} \times \R^2;  \,  v \cdot n_H^{ext} \leq -1   \right\}, \nonumber \\
&& \gamma^{2-}:= \left\{ (x,v) \in \partial H_{\delta} \times \R^2;  \, v \cdot n_H^{ext}  \leq - \frac{3}{2}   \right\}, \nonumber \\
&& \gamma^{3-}:= \left\{ (x,v) \in \partial H_{\delta} \times \R^2;  \, v \cdot n_H^{ext}  \leq - 2   \right\}, \label{eq:gamma3} \\
&& \gamma^{+}:= \left\{ (x,v) \in \partial H_{\delta} \times \R^2;  \, v \cdot n_H^{ext}  \geq 0    \right\}, \nonumber 
\end{eqnarray} where $n_H^{ext}$ is $\pm n_H$, taken in the outward direction with respect to $\partial H_{2\delta}$. It can be shown that 
\begin{equation}
\dist\left( [\partial H_{\delta} \times \R^2] \setminus \gamma^{2-}; \gamma^{3-}  \right) >0.
\label{eq:separationgamma}
\end{equation} Consequently, we may choose an absorption function $\A \in \C^{\infty} \cap W^{1,\infty} (\partial H_{\delta} \times \R^2;\R^+)$ such that 
\begin{eqnarray}
0 \leq \A(x,v)\leq 1, && \forall (x,v) \in \partial H_{\delta} \times \R^2, \label{eq:absortionfunction} \\
\A(x,v) = 1, && \forall (x,v) \in [\partial H_{\delta} \times \R^2 ]\setminus \gamma^{2-}, \nonumber  \\
\A(x,v) = 0, && \forall (x,v) \in \gamma^{3-}. \nonumber 
\end{eqnarray} We also choose a truncation function $\Y \in \C^{\infty}(\R^+;\R^+)$ satisfying
\begin{eqnarray}
\Y(t) = 0, && \forall t\in \left[0,\frac{T}{48}\right] \cup \left[\frac{47T}{48}, T \right],\nonumber \\
\Y(t) = 1, && \forall t\in \left[\frac{T}{24},\frac{23T}{24}\right].\nonumber 
\end{eqnarray} To give a sense to the procedure of absorption we need first the following result, which asserts that the number of times the characteristics associated to the Navier-Stokes velocity field of the previous part meet $\gamma-$ is finite.

\begin{lemma}
Let $g\in \S_{\epsilon}$ and let $u^g$ be given by (\ref{eq:NavierStokesforg}) accordingly. Let $(X^g,V^g)$ be the characteristics associated to the field $-v + u^g$. Then, for any $(x,v) \in \T^2 \times \R^2$, there exists $n(x,v) \in \N$ such that there exist $0 <t_1< \dots < t_{n(x,v)} <T$ such that
\begin{eqnarray}
&& \,\,\,  \left\{ (X^g,V^g)(t,0,x,v); \, t\in[0,T]   \right\} \cap \gamma^- = \left\{ t_i \right\}_{i=1}^{n(x,v)}, \label{eq:finitetimesofimpact}  \\
&& \, \,\, \exists s >0 \textrm{ s.t. } ( t_i - s, t_i +s) \cap (t_j - s, t_j + s) = \emptyset, \, \forall i,j =1, \dots, n(x,v), \label{eq:timesareisolated}
\end{eqnarray} with the convention that $n(x,v) = 0 $ and $\left\{ t_i \right\}_{i=1}^{n(x,v)} = \emptyset$ if $\left\{ (X^g,V^g)  \right\} \cap \gamma^- = \emptyset$.
\label{lemma:finitenumberofimpacts}
\end{lemma} For more details on this result, see \cite[p.348]{Glass} and \cite[p.5468]{GDHK1}. In the friction case, this holds true without further modification. \par

The previous lemma allows to define the following quantities. Let $f_0 \in \C^1(\T^2 \times \R^2)$ and let $(x,v) \in \T^2\times \R^2$. Then, for every $t_i, $ with $ i = 1,\cdots,n(x,v)$, we have $(\tilde{x},\tilde{v})=(X^g,V^g)(t_i,0,x,v) \in \gamma^-$. Moreover, let 
\begin{eqnarray}
&& f(t^-,\tilde{x},\tilde{v}) = \lim_{t\rightarrow t_i^-} f_0((X^g,V^g)(0,t,x,v)), \label{eq:fminus} \\
&& f(t^+,\tilde{x},\tilde{v}) = \lim_{t\rightarrow t_i^+} f_0((X^g,V^g)(0,t,x,v)). \label{eq:fplus}
\end{eqnarray}

We define $f:=\tilde{\V}_{\epsilon}[g]$ to be the solution of 
\begin{equation}
\left\{ \begin{array}{ll}
\d_t f + v\cdot \nabla_x f + u^g \cdot \nabla_v f - \div_v(vf) = 0, & (t,x,v) \in [0,T]\times [\T^2 \times \R^2]\setminus \gamma^{2-}, \\
f(0,x,v) = f_0(x,v), &  (x,v) \in \T^2 \times \R^2, \\
f(t^+,x,v) = (1 - \Y(t))f(t^-,x,v) + \Y(t)\A(x,v) f(t^-,x,v), & (t,x,v) \in [0,T] \times \gamma^-. 
\end{array} \right.
\label{eq:absortion}
\end{equation} Let us explain how the absorption procedure works. From (\ref{eq:champLipschitz}), the characteristics associated to the field $-v + u^g$ are regular. Thus, outside $\omega$, the system above defines a function $\tilde{\V}_{\epsilon}[g]$ of class $\C^1$. Moreover, the exact value of $\tilde{\V}_{\epsilon}[g]$ is given by these characteristics through (\ref{eq:explicitsolutionVNS}) and (\ref{eq:characteristics}). When the characteristics $(X^g, V^g)$ meet $\gamma^- $ at time $t$, $f(t^+,\cdot,\cdot)$ is fixed according to the last equation in (\ref{eq:absortion}). We can see the function $\Y (t) \A(x,v)$ as an opacity factor depending on time and on the incidence of the characteristics on $\partial H_{\delta}$. Indeed, $f(t^+,\cdot,\cdot)$ can take values varying from $f(t^-,\cdot,\cdot)$, in the case of no absorption, to $0$, according to the angle of incidence, the modulus of the velocity and time.
\par


\subsection{Extension} The function $\tilde{\V}_{\epsilon}[g]$ is not necessarily continuous around $[0,T] \times \gamma^{-} \subset [0,T] \times H_{\delta}$. To avoid this problem we shall use some extension operators preserving regularity. \par 
Let us first consider a linear extension operator 
\begin{equation*} 
\ovl{\pi}:\C^0(\T^2\setminus H_{\delta}) \rightarrow \C^0(\T^2),
\end{equation*} such that for any $\sigma \in (0,1)$, a $\C^{0,\sigma}$ function is mapped onto a $\C^{0,\sigma}$ function. This allows to define another linear extension operator by
\begin{equation*}
\begin{array}{cccc}
\pi : & \C^0([0,T] \times [\T^2 \setminus H_{\delta} ] \times \R^2 ) & \rightarrow & \C^0([0,T]\times \T^2 \times \R^2) \\
             & f & \mapsto & \pi f(t,x,v) = \ovl{\pi}\left[ f(t,\cdot, v)\right](x).
\end{array}
\end{equation*} Thus, $\pi$ is an extension satisfying the following properties: for every $f \in \C^0([0,T] \times (\T^2 \setminus H_{\delta}) \times \R^2)$, we have 
\begin{align}
&  \exists C_{\pi}>0 \textrm{ such that } \label{eq:extensioninfty} \\ 
&\|(1+|v|)^{\gamma+2} \pi(f) \|_{L^{\infty}(Q_T)} \leq C_{\pi}\|(1+|v|)^{\gamma +2} f \|_{L^{\infty}([0,T]\times (\T^2\setminus H_{\delta}) \times \R^2)}, \nonumber  \\
& \forall \sigma \in (0,1), \, \exists C_{\pi,\sigma}>0 \textrm{ such that } \label{eq:extensionHolder} \\ 
&  \|\pi(f) \|_{\C^{0,\sigma}(Q_T)} \leq C_{\pi,\sigma}\|(1+|v|)^{\gamma +2} f \|_{\C^{0,\sigma}([0,T]\times (\T^2\setminus H_{\delta}) \times \R^2)}. \nonumber 
\end{align} We introduce another truncation in time. Let $\tilde{\Y} \in \C^{\infty}(\R^+;[0,1])$ such that
\begin{equation}
\begin{array}{ll}
\tilde{\Y}(t) =0, & t\in \left[0,\frac{T}{100}\right], \\
\tilde{\Y}(t) =1, & t\in \left[\frac{T}{48},T\right].
\end{array} 
\end{equation} Finally, we set
\begin{equation}
\begin{array}{cccc}
\Pi: & \C^0([0,T]\times (\T^2 \setminus H_{\delta}) \times \R^2) & \rightarrow & \C^0([0,T]\times \T^2 \times \R^2), \\
& f & \mapsto & \Pi f = (1 - \tilde{\Y}(t))f + \tilde{\Y}(t) \pi f.
\end{array}
\label{eq:definitionofpi}
\end{equation} This allows to define the fixed point operator by
\begin{equation}
\V_{\epsilon}[g]:= \ovl{f} + \Pi\left( \tilde{\V}_{\epsilon}[f]_{|([0,T]\times (\T^2 \setminus H_{\delta}) \times \R^2)\cup([0,\frac{T}{48}] \times \T^2 \times \R^2)} \right), 
\label{eq:fixedpointoperator}
\end{equation} for every $(t,x,v) \in [0,T] \times \T^2 \times \R^2$.

\subsection{Existence of a fixed point}

The goal of this section is to prove the following result. 

\begin{proposition}
Let $T>T_0$, where $T_0>0$ is given by Proposition \ref{proposition:referencetrajectory} and let $c_1,c_2,c_3$ be large enough positive constants. Then, there exists $\epsilon_0>0$ such that, for any $\epsilon \leq \epsilon_0$, the operator defined by (\ref{eq:fixedpointoperator}) in the domain $\mathscr{S}_{\epsilon}$ defined by (\ref{eq:domainoftheoperator}) has a fixed point $g^*\in \mathscr{S}_{\epsilon}$. Furthermore, if $u^{g^*}$ denotes the solution of (\ref{eq:NavierStokesforg}) associated to $g^*$, the pair $(g^*,u^{g^*})$ is a strong solution of (\ref{eq:VNS}), with initial data $f_0$ and $u_0$, for a certain source term $G\in \C^0([0,T]\times \T^2\times \R^2)$.  
\label{proposition:fixedpoint}
\end{proposition} We shall carry out the proof of this result in several steps. The main idea is to apply the Leray-Schauder fixed-point theorem. To do this, we have to verify that 
\begin{enumerate}
\item The set $\S_{\epsilon}$ is convex and compact in $\C^0(Q_T)$,
\item $\V_{\epsilon}: \S_{\epsilon} \subset \C^0(Q_T) \rightarrow \C^0(Q_T) $ is continuous,
\item $\V_{\epsilon}(\S_{\epsilon}) \subset \S_{\epsilon}$.
\end{enumerate} 

The first point is straightforward, since the convexity of $\S_{\epsilon}$ is clear and the compactness is a consequence of Ascoli's theorem. The second point is similar to \cite[Section 3.3]{Glass} and holds without further modification, thanks to Lemma \ref{lemma:finitenumberofimpacts} and Lemma \ref{lemma:Gronwall}. \par

We need to show that point (3) holds. In other words, we have to prove that, for any $g \in \S_{\epsilon}$, $\V_{\epsilon}[g] \in \S_{\epsilon}$, i.e, points (a)--(c).

\subsubsection{Proof of point (b)} 
At this stage, we shall need the following property for the backwards characteristics associated to $-v + u^g$.

\begin{lemma}
Let $g \in \S_{\epsilon}$ and let $(X^g,V^g)$ be the characteristics associated to the field $-v + u^g$, according to (\ref{eq:NavierStokesforg}) and Proposition \ref{proposition:existenceNSg}. Then, there exists a constant $K_2=K_2(T,\gamma)>0$, independent of $g$, such that
\begin{equation}
\left|e^t|v| - |V^g(0,t,x,v)| \right| \leq  K_2,
\label{eq:pointb}
\end{equation} for any $(t,x,v) \in [0,T] \times \T^2 \times \R^2$. 
\label{lemma:b}
\end{lemma} 

\begin{proof}
By (\ref{eq:characteristics}), we have
\begin{align}
& \left|e^t|v| - |V^g(0,t,x,v)| \right| \nonumber \\
& \quad \quad \quad \leq \left| V^g(0,t,x,v) - e^tv   \right| \nonumber \\
& \quad \quad \quad = \left| \int_0^t e^{t-s} u^g(s, X^g(0,s,x,v)) \ud s \right| \nonumber \\
& \quad \quad \quad \leq C(T) \int_0^t \|u^g(s)\|_{L^{\infty}(\T^2)} \ud s \nonumber \\
& \quad \quad \quad \leq C(T) \|u^g\|_{L^2(0,T;L^{\infty}(\T^2))} \leq C(T)K_1, \nonumber
\end{align} using (\ref{eq:champuniforme}). This allows to conclude, choosing $K_2\geq C(T) K_1(T,M,\ovl{f})$.
\end{proof}

By construction of $\V_{\epsilon}$, we have 
\begin{align}
& \|(1+|v|)^{\gamma+2}\left( \V_{\epsilon}[f] - \ovl{f} \right) \|_{L^{\infty}(Q_T)} \nonumber \\
& \quad \quad = \left\|(1+|v|)^{\gamma+2} \Pi\left( \tilde{\V_{\epsilon}}[g]_{ |\left([0,T]\times (\T^2 \setminus H_{\delta}) \times \R^2 \cup [0,\frac{T}{48}] \times \T^2 \times \R^2 \right)}  \right)  \right\|_{L^{\infty}(Q_T)} \label{eq:majorationb} \\
& \quad \quad   \leq C_{\pi} \left\| (1+|v|)^{\gamma+2} \tilde{\V_{\epsilon}}[g]  \right\|_{L^{\infty}(Q_T)}, \nonumber 
\end{align} where we have used (\ref{eq:extensioninfty}). Moreover, by (\ref{eq:absortion}) and (\ref{eq:absortionfunction}), 
\begin{equation*}
|f(t^+,x,v)| \leq |f(t^-,x,v)|,
\end{equation*} which implies, through (\ref{eq:explicitsolutionVNS}),
\begin{equation*}
|\tilde{\V}_{\epsilon}[g](t,x,v)| \leq \left| e^{2t} f_0 \left( (X^g,V^g)(0,t,x,v) \right)  \right|.
\end{equation*} On the other hand,
\begin{align}
& \left|f_0\left( (X^g,V^g)(0,t,x,v) \right)\right| \nonumber \\
& =  \left( \frac{1+|V^g(0,t,x,v)|}{1+|V^g(0,t,x,v)|}  \right)^{\gamma+2} \left| f_0\left( (X^g,V^g)(0,t,x,v) \right)   \right| \nonumber \\
&\leq  \frac{\|(1+|v|)^{\gamma +2} f_0 \|_{L^{\infty}(Q_T)}}{ \left(  1 + |V^g(0,t,x,v)|  \right)^{\gamma+2}} \label{eq:majorationVg} \\
& = \frac{\|(1+|v|)^{\gamma +2} f_0 \|_{L^{\infty}(Q_T)}}{\left( 1 + \left[e^t|v| - (e^t|v| - |V^g(0,t,x,v)|) \right]  \right)^{\gamma+2}}  \nonumber \\
& \leq  \frac{ \left( 1 + \left| e^t|v| - |V^g(0,t,x,v)| \right| \right)^{\gamma+2} \|(1+|v|)^{\gamma +2} f_0 \|_{L^{\infty}(Q_T)} }{(1 + e^t|v|)^{\gamma +2}} \nonumber \\
& \leq  \frac{(1 + K_2(T,\gamma))^{\gamma + 2}  \|(1+|v|)^{\gamma +2} f_0 \|_{L^{\infty}(Q_T)}}{(1 + e^t|v|)^{\gamma+2}}, \nonumber
\end{align} where we have used (\ref{eq:pointb}) and the inequality (see \cite[Eq. (3.33), p. 347]{Glass}. 
\begin{equation}
\frac{1}{1+ |x-x'|} \leq \frac{1+|x'|}{1+|x|}, \quad  \forall x,x'\in \R^2.
\label{eq:inegalitequoitient}
\end{equation} Furthermore, since 
\begin{equation*}
(1+|v|)^{\gamma +2} |\tilde{\V_{\epsilon}}[g](t,x,v)| \leq (1 + e^t|v|)^{\gamma+2} |\tilde{\V_{\epsilon}}[g](t,x,v)|, 
\end{equation*} for every $(t,x,v)\in [0,T] \times \T^2 \times \R^2$, we have
\begin{align}
& \|(1+|v|)^{\gamma+2} \tilde{\V}_{\epsilon}[g] \|_{L^{\infty}(Q_T)} \label{eq:vtildepointb} \\ 
& \quad \quad \quad \quad \quad \leq e^{2T}\left( 1 + K_2(T,\gamma)\right)^{\gamma+2} \left(\|f_0\|_{\C^1} + \|(1+|v|)^{\gamma+2} f_0\|_{L^{\infty}} \right). \nonumber
\end{align} This gives that $\V_{\epsilon}[g]$ satisfies point (b), thanks to (\ref{eq:majorationb}) and choosing
\begin{equation}
c_1 \geq C_{\pi} e^{2T}\left( 1 + K_2(T,\gamma)\right)^{\gamma+2}.
\label{eq:choiceofc1}
\end{equation}


\subsubsection{Proof of point (c)} We need the following technical result, which can be adapted from \cite[Lemma 2, p. 347]{Glass}, thanks to Lemma \ref{lemma:Gronwall} and (\ref{eq:champLipschitz}).

\begin{lemma}
For any $g\in \S_{\epsilon}$, one has $\tilde{\V}_{\epsilon}[g] \in \C^1(Q_T \setminus \Sigma_T)$, with $\Sigma_T := [0,T] \times \gamma^-$. Moreover, there exists a constant $K_3=K_3(\gamma,\omega)>0$ such that
\begin{equation*}
\frac{\left| \tilde{\V}_{\epsilon}[g](t,x,v) - \tilde{\V}_{\epsilon}[g](t',x',v')  \right|}{(1+|v|)|(t,x,v) - (t',x',v')|} \leq K_3(\|f_0\|_{\C^1(\T^2\times \R^2)} + \| (1+|v|)^{\gamma+2} f_0 \|_{L^{\infty}(Q_T)}),
\end{equation*} for any $(t,x,v),(t',x',v') \in [0,T]\times (\T^2 \setminus \omega) \times \R^2$ with $|v-v'|<1$. Furthermore, if $f_0$ satisfies (\ref{eq:conditionforuniqueness}), we also have
\begin{eqnarray}
&&\|(1+|v|)^{\gamma+1}\nabla_{x,v}\tilde{\V}_{\epsilon}[g]\|_{L^{\infty}} \nonumber \\
&& \quad \quad \quad \quad  \leq K_4 \left( \|(1+|v|)^{\gamma+1}\nabla_{x,v}f_0 \|_{L^{\infty}(Q_T)} + \| (1+|v|)^{\gamma + 2} f_0 \|_{L^{\infty}} \right), \label{eq:estimeeunicite}
\end{eqnarray} for some constant $K_4=K_4(\kappa, g)>0$.
\label{lemma:pointc}
\end{lemma} 

Let $\delta_2$ be given by (\ref{eq:parametersdelta}). Again, by construction of $\V_{\epsilon}$ and (\ref{eq:extensionHolder}), we deduce
\begin{equation}
\| \V_{\epsilon}[g] - \ovl{f} \|_{\C^{0,\delta_2}(Q_T)} \leq C_{\pi,\delta_2} \| \tilde{\V}_{\epsilon}[g] \|_{\C^{0,\delta_2}([0,T] \times (\T^2 \setminus H_{\delta} ) \times \R^2)}.
\label{eq:majorationc}
\end{equation} Then, interpolating (\ref{eq:vtildepointb}) and Lemma \ref{lemma:pointc}, we have
\begin{eqnarray}
&& \frac{|\tilde{\V}_{\epsilon}[g](t,x,v) - \tilde{\V}_{\epsilon}[g](t',x',v')|}{|(t,x,v) - (t',x',v')|^{\delta_2}} \nonumber \\
&& = \left( \frac{|\tilde{\V}_{\epsilon}[g](t,x,v) - \tilde{\V}_{\epsilon}[g](t',x',v')|}{(1+|v|)|(t,x,v) - (t',x',v')|}  \right)^{\frac{\gamma +2}{\gamma +3}} \nonumber \\
&& \quad \quad \quad \quad \quad \quad \quad \quad \times \left( (1 + |v|)^{\gamma+2} |\tilde{\V}_{\epsilon}[g](t,x,v) - \tilde{\V}_{\epsilon}[g](t',x',v')| \right)^{1 - \frac{\gamma+2}{\gamma+3}} \nonumber \\
&& \leq K_3^{\frac{\gamma +2}{\gamma +3}}K_5^{1-\frac{\gamma +2}{\gamma +3}} \left( \|f_0\|_{\C^1(\T^2\times \R^2)} + \| (1+|v|)^{\gamma+2} f_0  \|_{L^{\infty}(Q_T)} \right), \nonumber 
\end{eqnarray} with 
\begin{equation*}
K_5 = e^{2T}\left( 1 + K_2(T,\gamma)\right)^{\gamma+2}.
\end{equation*} Whence, by (\ref{eq:majorationc}), this gives that $\tilde{\V}_{\epsilon}[g]$ satisfies point (c), choosing
\begin{equation}
c_2 \geq C_{\pi,\delta_2}K_3(\gamma,\omega)^{\frac{\gamma + 2}{\gamma + 3}} K_5(T,\gamma)^{1 - \frac{\gamma + 2}{\gamma + 3}}.
\label{eq:choiceofc2}
\end{equation}


\subsubsection{Proof of point (a)} We show first the $L^{\infty}$ estimate. Using the fact that $\rho_{\ovl{f}}=0$ and point (b), we find 
\begin{eqnarray}
\left\| \int_{\R^2} \left( \V_{\epsilon}[g] - \ovl{f} \right)  \ud v \right\|_{L^{\infty}(\Omega_T)} &=& \left\| \int_{\R^2} \left( \V_{\epsilon}[g](t,x,v)  \right)  \ud v \right\|_{L^{\infty}(\Omega_T)} \nonumber \\
& \leq &\sup_{t,x \in \Omega_T}  \int_{\R^2} \left|\V_{\epsilon}[g] (t,x,v) \right|  \ud v \nonumber \\
& \leq & K_6 (\|f_0\|_{\C^1} + \|(1+|v|)^{\gamma +2} f_0 \|_{L^{\infty}}),  \nonumber 
\end{eqnarray} with
\begin{equation*}
K_6 := c_1 \int_{\R^2} \frac{\ud v}{(1 + |v|)^{\gamma + 2}}.
\end{equation*} To show the H\"{o}lder estimate, we interpolate (\ref{eq:vtildepointb}) and (c). Indeed, if $\delta_1$ is given by (\ref{eq:parametersdelta}) and $\tilde{\gamma}:= 2 + \frac{\gamma }{2}$, we have 
\begin{eqnarray}
&& (1+|v|)^{\tilde{\gamma}}\frac{|\V_{\epsilon}[g](t,x,v) - \V_{\epsilon}[g](t',x',v)|}{|(t,x,v) - (t',x',v)|^{\delta_1}} \nonumber \\
&& = \left( (1 + |v|)^{\gamma+2} |\V_{\epsilon}[g](t,x,v) -\V_{\epsilon}[g](t',x',v)| \right)^{\frac{1}{2} + \frac{1}{\gamma +2}} \nonumber \\
&& \quad \quad \quad \quad \quad \quad \quad \times \left( \frac{|\V_{\epsilon}[g](t,x,v) - \V_{\epsilon}[g](t',x',v)|}{|(t,x,v) - (t',x',v)|^{\delta_2}} \right)^{\frac{1}{2} - \frac{1}{\gamma +2}} \nonumber \\
&& \leq c_1^{\frac{1}{2} + \frac{1}{\gamma +2 }} c_2^{\frac{1}{2} - \frac{1}{\gamma + 2}} \left( \|f_0\|_{\C^1(\T^2\times \R^2)} + \| (1+|v|)^{\gamma+2} f_0   \|_{L^{\infty}(Q_T)}  \right). \nonumber 
\end{eqnarray} Consequently, choosing
\begin{equation}
c_3 \geq K_6 + c_1^{\frac{1}{2} + \frac{1}{\gamma +2 }} c_2^{\frac{1}{2} - \frac{1}{\gamma + 2}},
\label{eq:choiceofc3}
\end{equation} and thanks to (\ref{eq:smalldata}), we have that $\V_{\epsilon}[g]$ satisfies point (a). \par

\begin{proof}[Proof of Proposition \ref{proposition:fixedpoint}]

Let us choose $c_1,c_2,c_3$ large enough so that (\ref{eq:choiceofc1}), (\ref{eq:choiceofc2}) and (\ref{eq:choiceofc3}) are satisfied. Let us choose $\epsilon_0$ sufficiently small, given by (\ref{eq:epsilon0}). Then, the smallness assumption (\ref{eq:smalldata}) and the properties of $\V_{\epsilon} $ and $\Pi$ allow to conclude that if $\epsilon \leq \epsilon_0$, then
\begin{equation*}
\V_{\epsilon}(\S_{\epsilon}) \subset \S_{\epsilon}.
\end{equation*} Thus, thanks to the Leray-Schauder theorem, there exists $g \in \S_{\epsilon}$ such that $\V_{\epsilon}[g] = g$. This ends the proof of Proposition \ref{proposition:fixedpoint}.  

\end{proof}

We can furthermore obtain a regularity result for the fixed point found above.

\begin{corollary}
let $\epsilon \leq \epsilon_0$ and $g=\V_{\epsilon}[g]$. If $f_0\in \C^1(\T^2 \times \R^2)$ satisfies (\ref{eq:conditionforuniqueness}), then
\begin{align}
& g\in \C^1(Q_T), \label{eq:regulariteunicite} \\ 
& \exists \kappa'>0, \, \sup_{(t,x,v) \in Q_T} (1 + |v|)^{\gamma+1}\left( |g| + |\nabla_{x,v}g| \right)(t,x,v) \leq \kappa'. \label{eq:classeunicite} 
\end{align}
\label{corollary:regularityofg}
\end{corollary}

\begin{proof}
This is a consequence of the construction and Lemma \ref{lemma:pointc}, which gives (\ref{eq:estimeeunicite}). By the construction of $\ovl{f}$ and $\tilde{\V}_{\epsilon}$, and since $\Pi$ preserves regularity, we deduce (\ref{eq:regulariteunicite}) and (\ref{eq:classeunicite}).
\end{proof}


\section{Stability estimates for the Vlasov-Navier-Stokes system}
\label{sec: stability VNS}

The goal of this section is to prove the following stability estimate for the strong solutions of (\ref{eq:VNS}), that will be crucial in the proof of Theorem \ref{thm:controllability}, both for the controllability as well as for uniqueness.

\begin{proposition}
Let $T>0$, $\gamma > 2$ and let $(g,u^g)$ and $(f,u^f)$ be two strong solutions of system (\ref{eq:VNS}), according to Definition \ref{definition:strong}, for initial data
\begin{align*}
u^g|_{t=0}=u^g_0, \quad g|_{t=0} = g_0, \\
u^f|_{t=0}=u^f_0, \quad f|_{t=0} = f_0.
\end{align*} Assume further that there exists $\kappa >0 $ such that
\begin{equation}
\sup_{(t,x,v) \in Q_T} \left( 1 + |v| \right)^{\gamma + 1} \left( |f| + |g| + |\nabla_x f| + |\nabla _x g|  \right) \leq \kappa.
\label{eq:hypothesestabilite}
\end{equation} Then, there exists a constant $C>0$ such that for any $t\in (0,T]$,
\begin{align}
& \|(u^g - u^f)(t)\|_{H^{\frac{1}{2}}(\T^2)}^2 + \int_0^t \|\nabla(u^g - u^f)(s)\|_{H^{\frac{1}{2}}(\T^2)}^2 \ud s \label{eq:stabilityH32}  \\
& \quad \quad \quad \quad  \leq e^{C(t)} \left( \|u^g_0 - u^f_0 \|^2_{H^{\frac{1}{2}}(\T^2)}  + \int_0^t \|j_{g-f} - \rho_{g-f} u^f  \|^2_{L^2(\T^2)}   \right)  \nonumber,
\end{align} where 
\begin{equation*}
C(t):= \int_0^t \left( 1 + \|\rho_f(s)\|_{L^{\infty}(\T^2)} + \|\rho_f(s)\|_{L^{\infty}(\T^2)}^2 \right) \ud s. 
\end{equation*}
\label{proposition:stabilityH32}
\end{proposition} 

\begin{proof}
We observe that the difference $w:= u^g - u^f$ satisfies the system
\begin{equation}
\left\{ \begin{array}{ll}
\partial_t w - \Delta w + \nabla \pi = -u^g\cdot \nabla u^g + u^f \cdot \nabla u^f   +j_{g - f} - \rho_g u^g + \rho_f u^f,  & \textrm{in } \Omega_T, \\
\div_x w = 0,    &  \textrm{in } \Omega_T, \\
w|_{t=0} = u^g_0 - u^f_0, & \textrm{in } \T^2.
\end{array} \right.
\label{eq:NSdifference}
\end{equation} Thanks to (\ref{eq:champregulier}), the difference $u^g - u^f$ belongs at least to $\C^0([0,T]; \V_{\sigma})$, which allows, up to a regularisation in time argument, to use it as a test function. Indeed, multiplying the equation in (\ref{eq:NSdifference}) by $ u^g - u^f$ and taking the $H^{\frac{1}{2}}$ scalar product, we find
\begin{align}
& \frac{\ud }{\ud t} \| (u^g - u^f)(t) \|^2_{H^{\frac{1}{2}}} + \| \nabla (u^g - u^f)(t)\|^2_{H^{\frac{1}{2}}}  \nonumber \\
& \quad \quad \quad  \quad  = \left|  \left( -u^g\cdot \nabla u^g + u^f\cdot \nabla u^f + j_{g - f} - \rho_g u^g + \rho_f u^f , u^g - u^f    \right)_{H^{\frac{1}{2}}}(t) \right|  \nonumber  \\
& \quad \quad \quad  \quad \leq \mathscr{T}_1(t) + \mathscr{T}_2(t), \label{eq:energyestimatedifference} 
\end{align} with 
\begin{align}
\mathscr{T}_1(t)  & := \left|  \left( -u^g\cdot \nabla u^g + u^f\cdot \nabla u^f, u^g - u^f  \right)_{H^{\frac{1}{2}}}(t) \right|, \label{eq:stabilityT1} \\
\mathscr{T}_2(t)  & := \left|  \left(  j_{g - f} - \rho_g u^g + \rho_f u^f, u^g - u^f    \right)_{H^{\frac{1}{2}}}(t) \right|.
\end{align} We then have to estimate the two terms of the right-hand side. \par 

\vspace{0.5em}
\textbf{First Term.} For the term $\mathscr{T}_1$, according to (\ref{eq:normsFourier}), one has
\begin{align}
& \mathscr{T}_1(t) = \left| \sum_{n\in \Z^2} (-u^g\cdot \nabla u^g + u^f \cdot \nabla u^f)_n (1+|n|^2)^{\frac{1}{2}} \overline{ (u^g - u^f)_n } \right| \label{eq:convection12} \\
& \quad \quad = \left| \sum_{n\in \Z^2 \setminus\left\{ 0 \right\}} (-u^g\cdot \nabla u^g + u^f \cdot \nabla u^f)_n (1+|n|^2)^{\frac{1}{2}} \overline{ (u^g - u^f)_n } \right| \nonumber \\
& \quad \quad \lesssim \sum_{n\in \Z^2 \setminus\left\{ 0 \right\}} \left| (-u^g\cdot \nabla u^g + u^f \cdot \nabla u^f)_n \right| |n| \left|(u^g - u^f)_n \right| \nonumber \\
& \quad \quad \lesssim \left( \sum_{n\in \Z^2 \setminus\left\{ 0 \right\}}\left| (-u^g\cdot \nabla u^g + u^f \cdot \nabla u^f)_n \right|^2   \right)^{\frac{1}{2}} \left( \sum_{n\in \Z^2 \setminus\left\{ 0 \right\}} |n|^2 \left|(u^g - u^f)_n \right|^2   \right)^{\frac{1}{2}} \nonumber \\
& \quad \quad = \| -u^g\cdot \nabla u^g + u^f \cdot \nabla u^f \|_{L^2(\T^2)} \left( \sum_{n\in \Z^2 \setminus\left\{ 0 \right\}} |n|^2 \left|(u^g - u^f)_n \right|^2   \right)^{\frac{1}{2}}, \nonumber
\end{align} thanks to the Cauchy-Schwarz's inequality. \par 

For the second term in the above inequality, we have, for every $n\in \Z^2\setminus \left\{0 \right\}$, and for any $\eta>0$,
\begin{align*}
|n|^2 |(u^g-u^f)_n|^2 & \lesssim \frac{1}{\eta^2} |n||(u^g-u^f)_n|^2 + \eta^2 |n|^3 |(u^g-u^f)_n|^2 \\
& \lesssim \left( \frac{1}{\eta} |n|^{\frac{1}{2}} |(u^g-u^f)_n| + \eta |n|^{\frac{3}{2}} |(u^g-u^f)_n|  \right)^2
\end{align*} This gives, thanks to Minkowski's inequality,
\begin{align}
& \left( \sum_{n\in \Z^2\setminus\left\{0\right\}} |n|^2 |(u^g - u^f)_n|^2 \right)^{\frac{1}{2}} \label{eq:estimeediferenceconvection} \\
& \quad \quad \leq \left[ \sum_{n\in \Z^2\setminus\left\{0\right\}} \left( \frac{1}{\eta} |n|^{\frac{1}{2}} |(u^g-u^f)_n| + \eta |n|^{\frac{3}{2}} |(u^g-u^f)_n|  \right)^2       \right]^{\frac{1}{2}} \nonumber \\
& \quad \quad \lesssim \frac{1}{\eta} \left( \sum_{n\in \Z^2\setminus\left\{0\right\}} |n||(u^g-u^f)_n|^2  \right)^{\frac{1}{2}}
+ \eta \left( \sum_{n\in \Z^2\setminus\left\{0\right\}} |n|^3|(u^g-u^f)_n|^2  \right)^{\frac{1}{2}} \nonumber \\
& \quad \quad = \frac{1}{ \eta } \| u^g - u^f\|_{H^{\frac{1}{2}}} + \eta \| \nabla (u^g - u^f) \|_{H^{\frac{1}{2}}}. \nonumber
\end{align} On the other hand, in the same fashion as in (\ref{eq:convectionLr}) with $r=2$, we can estimate the convection term as follows.
\begin{align}
& \left\| \left( - u^g \cdot\nabla u^g + u^f \cdot\nabla u^f \right)(t) \right\|_{L^2(\T^2)} = \| (u^g - u^f )\cdot \nabla u^g + u^f \cdot\nabla (u^g - u^f) \|_{L^2(\T^2)} \nonumber \\
& \quad \quad \quad \leq \| (u^g - u^f)\cdot\nabla u^g\|_{L^2} + \|u^f \cdot\nabla (u^g - u^f) \|_{L^2} \nonumber \\
& \quad \quad \quad \lesssim \|u^g - u^f\|_{L^4}\|\nabla u^g\|_{L^4} + \|u^f\|_{L^4} \| \nabla (u^g - u^f) \|_{L^4}   \nonumber  \\
& \quad \quad \quad \lesssim  \left( \|\nabla u^g(t)\|_{L^4} + \|u^f(t)\|_{L^4}   \right) \left( \| (u^g - u^f)(t) \|_{H^{\frac{1}{2}}} + \| \nabla(u^g - u^f)(t) \|_{H^{\frac{1}{2}}} \right), \label{eq:stabilityT1premiermorceau}
\end{align} thanks to the Sobolev embedding $H^{\frac{1}{2}}(\T^2) \hookrightarrow L^4(\T^2)$. In order to estimate the last inequality, we shall prove that 
\begin{equation}
\sup_{t} \left(  \|\nabla u^g(t)\|_{L^4(\T^2)} + \|u^f(t) \|_{L^4(\T^2)} \right) \leq C,
\label{eq:stabilitymajorationinftyH2}
\end{equation} for some constant $C>0$. \par 
Indeed, since $u^f,u^g \in \C^0_tH_x^1 \cap L^2_tH^2_x $, according to (\ref{eq:champregulier}), and thanks to the hypothesis (\ref{eq:hypothesestabilite}), we have that 
\begin{equation*}
j_g - \rho_g u^g \in L^2(0,T;H^1(\T^2)), \quad  j_f - \rho_f u^f \in L^2(0,T;H^1(\T^2)).
\end{equation*} Thus, from classical parabolic regularity result for the Navier-Stokes system (see \cite[Theorem V.2.10 and Corollary V.2.11, p.384]{Boyer}), we have
\begin{equation}
u^g,u^f \in \C^0([0,T];H^2(\T^2)) \cap L^2(0,T;H^3(\T^2)).
\label{eq:stabilityregulariteplusplus}
\end{equation} Thus, thanks to the Sobolev embedding $H^{\frac{1}{2}}(\T^2) \hookrightarrow L^4(\T^2)$, this entails, for any $t\in (0,T]$, that
\begin{align*}
\| \nabla u^g(t) \|_{L^4(\T^2)} + \| u^f(t) \|_{L^4(\T^2)}  & \lesssim \| \nabla u^g(t) \|_{H^{\frac{1}{2}}(\T^2)} + \| u^f(t) \|_{H^{\frac{1}{2}}(\T^2)} \\
& \lesssim \max \left\{ \| u^g(t)\|_{H^2(\T^2)}, \| u^f(t)\|_{H^2(\T^2)} \right\} \\
& \lesssim \max \left\{ \| u^g\|_{\C^0_t H^2(\T^2)}, \| u^f\|_{\C^0_t H^2(\T^2)} \right\},
\end{align*} which is finite thanks to (\ref{eq:stabilityregulariteplusplus}). This proves (\ref{eq:stabilitymajorationinftyH2}). \par 

Thus, combining (\ref{eq:stabilitymajorationinftyH2}) with (\ref{eq:stabilityT1premiermorceau}), we obtain, for any $t\in (0,T]$,
\begin{align*}
& \int_0^t \left\| \left( - u^g \cdot\nabla u^g + u^f \cdot\nabla u^f \right)(s) \right\|_{L^2(\T^2)}^2 \ud s \\
& \quad \quad \quad \quad \quad \quad \quad \quad \quad  \lesssim \int_0^t \left( \| (u^g - u^f)(s) \|_{H^{\frac{1}{2}}}^2 + \| \nabla(u^g - u^f)(s) \|_{H^{\frac{1}{2}}}^2 \right) \ud s. 
\end{align*} Consequently, injecting this in (\ref{eq:convection12}) and taking (\ref{eq:estimeediferenceconvection}) into account, one has 
\begin{align*}
& \int_0^T \mathscr{T}_1(t) \ud t \\
&    \leq \int_0^T \|(u^g\cdot \nabla u^g - u^f\cdot \nabla u^f )(t) \|_{L^2} \left( \sum_{n\in \Z^2} |n|^2 \left| (u^g-u^f)_n(t) \right|^2  \right)^{\frac{1}{2}} \ud t \\ 
& \leq \left( \int_0^T  \|   (u^g\cdot \nabla u^g - u^f\cdot \nabla u^f )(t)\|^2_{L^2} \ud t  \right)^{\frac{1}{2}} \left( \int_0^T \sum_{n\in \Z^2} |n|^2 \left| (u^g-u^f)_n(t) \right|^2 \ud t  \right)^{\frac{1}{2}} \\
& \leq \left( \int_0^T  \|  (u^g- u^f)(t) \|^2_{L^2} +\| \nabla (u^g - u^f )(t)\|^2_{L^2} \ud t  \right)^{\frac{1}{2}} \\
& \quad \quad \quad \quad \quad \quad \quad \times \left( \int_0^T \frac{1}{\eta}\|  (u^g- u^f)(t) \|^2_{H^{\frac{1}{2}}} +  \eta \| \nabla (u^g - u^f )(t)\|^2_{H^{\frac{1}{2}}} \ud t  \right)^{\frac{1}{2}}.
\end{align*} Hence, choosing $\eta <1$ small enough, this yields 
\begin{equation}
\int_0^T \mathscr{T}_1(t) \ud t \lesssim  \frac{1}{\eta} \int_0^T\|  (u^g- u^f)(t) \|^2_{H^{\frac{1}{2}}} \ud t  + \sqrt{\eta }\int_0^T \| \nabla (u^g - u^f )(t)\|^2_{H^{\frac{1}{2}}} \ud t.
\label{eq:stabilitymajorationT1}
\end{equation}

\vspace{0.5em}

\textbf{Second term.} Let us treat next the other term in the right-hand side of (\ref{eq:energyestimatedifference}). We proceed as follows. For any $t\in [0,T]$, we have
\begin{align}
\mathscr{T}_2(t) &   \leq \left| \left(j_{g-f} -\rho_{g-f}u^g, u^g - u^f \right)_{H^{\frac{1}{2}}} \right| + \left| \left(  \rho_f(u^g - u ^f), u^g - u^f \right)_{H^{\frac{1}{2}}} \right| \label{eq:stabilitesecondterme} \\
& = \left| \sum_{n\in \Z^2} (j_{g-f} -\rho_{g-f} u^g )_n (1+|n|^2)^{\frac{1}{2}} \overline{ (u^g - u^f)_n} \right|  + \left| \left(  \rho_f(u^g - u ^f), u^g - u^f \right)_{H^{\frac{1}{2}}} \right| \nonumber \\
& \leq \sum_{n\in \Z^2 \setminus \left\{ 0 \right\}} \left|  (j_{g - f} -\rho_g u^g )_n (1+|n|^2)^{\frac{1}{2}} \overline{( u^g - u^f )_n}    \right| + \left|(j_{g-f} - \rho_g u^g)_0( u^g - u^f )_0 \right| \nonumber \\ 
& \quad \quad \quad \quad \quad \quad \quad \quad \quad \quad \quad  + \left| \left(  \rho_f(u^g - u ^f), u^g - u^f \right)_{H^{\frac{1}{2}}} \right| \nonumber \\
& \lesssim  \sum_{n\in \Z^2\setminus \left\{ 0 \right\}} \left| (j_{g - f} - \rho_g u^g )_n  \right|^2 +  \sum_{n\in \Z^2 \setminus \left\{ 0 \right\}} |n|^2 \left| ( u^g - u^f )_n  \right|^2 \nonumber \\
& \quad \quad \quad \quad \quad \quad  + \left| \int_{\T^2} (j_{g - f } - \rho_g u^g) \ud x \int_{\T^2} (u^g - u^f) \ud x \right| +  \left| \left(  \rho_f(u^g - u ^f), u^g - u^f \right)_{H^{\frac{1}{2}}} \right| \nonumber \\
& = \mathcal{B}_1(t) + \mathcal{B}_2(t) + \mathcal{B}_3(t), \nonumber
\end{align} with
\begin{align*}
& \mathcal{B}_1(t) :=  \sum_{n\in \Z^2\setminus \left\{ 0 \right\}} \left| (j_{g - f} - \rho_g u^g )_n  \right|^2 + \left| \int_{\T^2} (j_{g - f} - \rho_g u^g ) \ud x \int_{\T^2} ( u^g - u^f ) \ud x \right|, \\
& \mathcal{B}_2(t) :=  \sum_{n\in \Z^2 \setminus \left\{ 0 \right\}} |n|^2 \left| ( u^g - u^f )_n  \right|^2, \\
& \mathcal{B}_3(t):=  \left| \left(  \rho_f(u^g - u ^f), u^g - u^f \right)_{H^{\frac{1}{2}}} \right|.
\end{align*} For the first term, we write 
\begin{align}
\mathcal{B}_1 & \lesssim \| j_{g-f} - \rho_{g-f}u^g \|^2_{L^2} \label{eq:termeT1stabilite} \\ 
& \quad \quad \quad \quad + \frac{1}{2}\int_{\T^2} |j_{g-f} - \rho_{g-f}u^g |^2 \ud x + \frac{1}{2} \int_{\T^2}|u^g - u^f|^2 \ud x \nonumber \\
& \lesssim \frac{3}{2} \| j_{g-f} - \rho_{g-f}u^g \|^2_{L^2} + \|u^g - u^f\|^2_{H^\frac{1}{2}}. \nonumber 
\end{align} For the second term, using the Cauchy-Schwarz's inequality, one finds
\begin{align}
\mathcal{B}_2 & = \sum_{n\in \Z^2\setminus \left\{ 0 \right\}} |n|^{\frac{1}{2} + \frac{3}{2}} |(u^g-u^f)_n|^2  \label{eq:termeT2stabilite} \\
& \lesssim \left( \sum_{n\in \Z^2\setminus \left\{ 0 \right\}} |n||(u^g-u^f)_n|^2    \right)^{\frac{1}{2}} \left( \sum_{n\in \Z^2\setminus \left\{ 0 \right\}} |n|^3|(u^g-u^f)_n|^2    \right)^{\frac{1}{2}}  \nonumber  \\
& \lesssim \frac{1}{\eta} \sum_{n\in \Z^2\setminus \left\{ 0 \right\}} |n||(u^g-u^f)_n|^2  + \eta \sum_{n\in \Z^2\setminus \left\{ 0 \right\}} |n|^3|(u^g-u^f)_n|^2   \nonumber  \\
& \lesssim \frac{1}{\eta} \| u^g - u^f \|_{H^{\frac{1}{2}}}^2 + \eta \|\nabla(u^g - u^f)\|^2_{H^{\frac{1}{2}}}, \nonumber
\end{align} for some $\eta \in (0,1)$ to be chosen later on. For the third term, we have, thanks to (\ref{eq:hypothesestabilite}),
\begin{align}
\mathcal{B}_3 & \lesssim \sum_{n\in \Z^2\setminus \left\{ 0 \right\}} \left( \rho_f(u^g-u^f)\right)_n (1+ |n|^2)^{\frac{1}{2}} \overline{(u^g-u^f)_n} \label{eq:termeT3stabilite} \\
& \quad \quad \quad \quad  \quad \quad + \left|  \int_{\T^2} \rho_f (u^g-u^f)\ud x \int_{\T^2}(u^g-u^f)\ud x \right| \nonumber\\
& \lesssim \frac{1}{\eta}\sum_{n\in\Z^2 \setminus \left\{0\right\}}  \left|  \left( \rho_f(u^g-u^f)\right)_n \right|^2 +  \eta \sum_{n\in\Z^2 \setminus \left\{0\right\}} |n|^2 |(u^g-u^f)_n|^2 \nonumber \\
& \quad \quad \quad \quad \quad \quad + \|\rho_f\|_{L^{\infty}(\T^2)} \left( \int_{\T^2} |u^g-u^f|\ud x \right)^2 \nonumber \\
& \lesssim \left( \frac{\|\rho_f\|_{L^{\infty}(\T^2)}}{\eta} + 1 \right) \|\rho_f\|_{L^{\infty}(\T^2)} \|u^g-u^f\|^2_{L^2(\T^2)} + \eta \|\nabla(u^g-u^f)\|^2_{L^2(\T^2)}, \nonumber
\end{align} thanks to Jensen's inequality. \par 

We finally have, for any $t\in (0,T]$,
\begin{align}
\mathscr{T}_2(t) & \leq \left( 1 + \|\rho_f(t) \|_{L^2(\T^2)} + \frac{1}{\eta}(1 + \|\rho_f(t) \|_{L^2(\T^2)}^2) \right) \|(u^g - u^f)(t)\|_{H^{\frac{1}{2}}(\T^2)}^2 \label{eq:stabilitymajorationT2} \\
& \quad \quad \quad \quad \quad \quad + \eta \| \nabla(u^g - u^f)(t) \|^2_{H^{\frac{1}{2}}(\T^2)} + \| (j_{g-f} - \rho_{g-f}u^g)(t) \|^2_{L^2(\T^2)}. \nonumber 
\end{align}

\vspace{0.5em}

\textbf{Conclusion.} Integrating (\ref{eq:energyestimatedifference}) with respect to time yields
\begin{align*}
& \| (u^g - u^f)(t) \|_{H^{\frac{1}{2}}}^2 +  \int_0^t \| \nabla (u^g - u^f)(s) \|_{H^{\frac{1}{2}}}^2 \ud s   \\ 
& \quad \quad \quad \quad \quad \quad  \quad \leq \|u^g_0 - u^f_0\|^2_{H^{\frac{1}{2}}} + \int_0^t \mathscr{T}_1(s) \ud s + \int_0^t \mathscr{T}_2(s) \ud s. 
\end{align*} Thus, from (\ref{eq:stabilitymajorationT1}) and (\ref{eq:stabilitymajorationT2}), and choosing $\eta \in (0,1)$ small enough, this gives 
\begin{align*}
& \|(u^g-u^f)(t)\|^2_{H^{\frac{1}{2}}} + \frac{1}{2} \int_0^t \|\nabla(u^g-u^f)(s)\|^2_{H^{\frac{1}{2}}} \ud s \\
& \quad \quad \quad \quad \lesssim \| u^g_0 - u^f_0 \|^2_{H^{\frac{1}{2}}}  + \int_0^t \|j_{g-f}(s) - \rho_{g-f}(s)u^g(s) \|^2_{L^2}  \ud s \\
& \quad \quad \quad \quad \quad  + \int_0^t \left( 1 + \|\rho_f(s) \|_{L^{\infty}} + \frac{1}{\eta}(1 + \|\rho_f(s)\|_{L^{\infty}}^2 )  \right) \| (u^g - u^f)(s) \|^2_{H^{\frac{1}{2}}} \ud s 
\end{align*} By Gronwall's inequality, we obtain (\ref{eq:stabilityH32}).

\end{proof}


\section{Proof of Theorem \ref{thm:controllability}. Controllability.}
\label{sec: fixedpointrelevantVNS}
Let us denote by $g$ the fixed point of the operator $\V_{\epsilon}$, found in Proposition \ref{proposition:fixedpoint}, for $\epsilon \leq \epsilon_0$. The purpose of the following result is to establish that, choosing $\epsilon \leq \epsilon_1$ possibly smaller, the characteristics associated to $-v + u^g$ meet $H_{\delta}$ with sufficient speed.

\begin{proposition}
There exists $\epsilon_1>0$ and $M_1>0 $ such that for any $\epsilon \leq \epsilon_1$ and $M\leq M_1$ the characteristics associated to $-v + u^g$, where $u^g$ is given by Proposition \ref{proposition:existenceNSg}, namely $(X^g,V^g)$, satisfy the following property
\begin{equation}
\forall (x,v) \in \T^2\times \R^2, \, \exists t\in \left[ \frac{T}{48}, \frac{47T}{48} \right] \textrm{ s.t. } (X^g,V^g)(t,0,x,v) \in \gamma^{3-}, \label{eq:characteristicsarerelevant} \\
\end{equation} where $\gamma^{3-}$ is defined in (\ref{eq:gamma3}).
\label{proposition:fixedpointisrelevant}
\end{proposition}

\begin{proof}
In order to prove (\ref{eq:characteristicsarerelevant}), we proceed in two steps. \par 

\vspace{0.5em}
\textit{Step 1. Stability argument.} We shall show that the characteristics $(X^g,V^g)$ are uniformly close to $(\ovl{X},\ovl{V})$ whenever $\epsilon$ and $M$ are chosen sufficiently small.   \par 

Indeed, from (\ref{eq:characteristics}), we have, for every $(t,x,v) \in Q_T$,
\begin{align}
& \ovl{X}(t,0,x,v) - X^g(t,0,x,v) \label{eq:XovlminusXg} \\
& \quad \quad \quad = \int_0^t \int_0^s e^{\tau - s} \left( \ovl{u}(\tau, \ovl{X}(\tau, 0, x, v)) - u^g(\tau, X^g(\tau,0,x,v))      \right) \ud \tau \ud s\nonumber \\
& \quad \quad \quad = \mathcal{J}_1 + \mathcal{J}_2, \nonumber 
\end{align} with
\begin{align*}
\mathcal{J}_1 & := \int_0^t \int_0^s e^{\tau - s}\left( \ovl{u}(\tau, \ovl{X}(\tau, 0, x, v) - \ovl{u}(\tau, X^g(\tau, 0, x, v) \right) \ud \tau \ud s,  \\
\mathcal{J}_2 & := \int_0^t \int_0^s e^{\tau - s}\left( \ovl{u}(\tau, X^g(\tau, 0, x, v) - u^g(\tau, X^g(\tau, 0, x, v) \right) \ud \tau \ud s. 
\end{align*} For the first term above, using (\ref{eq:regularityofthereferencefield}), we find
\begin{equation}
|\mathcal{J}_1| \leq C(T,\ovl{u}) \int_0^t \left| \ovl{X}(\tau,0,x,v) - X^g(\tau,0,x,v) \right| \ud \tau.
\label{eq:jota1estimate}
\end{equation} For the second term above, we have
\begin{equation}
|\mathcal{J}_2| \leq C(T) \| \ovl{u} - u^g \|_{L^2_t L^{\infty}_x}.
\label{eq:jota2estimate}
\end{equation} Consequently, we have to obtain a precise estimate of the difference $u^g -\ovl{u}$ in $L^2_t L^{\infty}_x$. \par 
In order to do this, we shall use Proposition \ref{proposition:stabilityH32} with the solutions $u^g$ and $\ovl{u}$. Let us observe that, thanks to (\ref{eq:conditionforuniqueness}), (\ref{eq:classeunicite}) and (\ref{eq:regularityofthereferencedistribution}), hypothesis (\ref{eq:hypothesestabilite}) is staisfied in this case. Thus, (\ref{eq:stabilityH32}) yields
\begin{align}
& \| (u^g - \ovl{u})(t) \|^2_{H^{\frac{1}{2}}} + \int_0^t \| \nabla (u^g - \ovl{u})(s) \|^2_{H^{\frac{1}{2}}} \ud s \label{eq:stabilitereference} \\
& \quad \quad \quad \quad \quad \quad \quad \lesssim \|u_0\|^2_{H^{\frac{1}{2}}} + \int_0^t \| j_{g - \ovl{f}} - \rho_gu^g \|^2_{L^2} \ud s, \nonumber
\end{align} thanks to (\ref{eq:referencezeroatendandbeginning}) and the fact that $\rho_{\ovl{f}} \equiv 0$. \par 
Firstly, we observe that, thanks to (\ref{eq:lemma412}),
\begin{equation*}
\| j_{g - \ovl{f}} \|^2_{L^2}  = \int_{\T^2} \left| \int_{\R^2} v (g - \ovl{f}) \ud v   \right|^2 \ud x \lesssim \epsilon^2 
\end{equation*} Secondly, using point (a) and (\ref{eq:estimationuniformeNSg}),
\begin{align*}
\sup_{t\in [0,T]} \|\rho_g (t) \|^2_{L^2(\T^2)} \| u^g(t)\|_{L^2(\T^2)}^2 & \leq 2c_3^2 \epsilon^2 e^T \left( M^2 + T(1 + \|j_{\ovl{f}}\|^2_{L^{\infty}_tL^2_x}) \right) \\
& \lesssim \epsilon^2.
\end{align*} Then, for any $t\in (0,T]$,
\begin{equation*}
\int_0^t \| j_{g-\ovl{f}} - \rho_g u^g \|_{L^2(\T^2)}^2 \ud s \leq \epsilon^2,
\end{equation*} Thus,  
\begin{align*}
\| (\ovl{u} - u^g )(t)\|_{H^{\frac{1}{2}}}^2 + \frac{1}{2}\int_0^t \| \nabla (\ovl{u} - u^g ) (s)\|^2_{H^{\frac{1}{2}}} \ud s & \lesssim \|u_0\|_{H^{\frac{1}{2}}}^2 + \epsilon^2 \\
&  \lesssim M^2 + \epsilon^2,
\end{align*} thanks to (\ref{eq:massechampinitiale}). To conclude, using the injection $H^{\frac{3}{2}}(\T^2) \hookrightarrow L^{\infty}(\T^2)$, we deduce
\begin{align*}
& \int_0^T \| (\ovl{u} - u^g)(t)\|_{L^{\infty}(\T^2)}^2 \ud t \\
& \quad \quad \quad \quad \lesssim \int_0^T \| (\ovl{u} - u^g)(t)\|^2_{H^{\frac{3}{2}}} \ud t \lesssim M^2 + \epsilon^2,
\end{align*} which gives the $L^2_t L^{\infty}_x$ estimate. \par 

We have, from (\ref{eq:XovlminusXg}), (\ref{eq:jota1estimate}), (\ref{eq:jota2estimate}) and the previous inequalities,
\begin{align*}
& \left| \ovl{X}(t,0,x,v) - X^g(t,0,x,v)   \right| \\
& \quad \quad \quad \quad \leq |\mathcal{J}_1| + |\mathcal{J}_2| \\
& \quad \quad \quad \quad \lesssim \int_0^t \left| \ovl{X}(s,0,x,v) - X^g(s,0,x,v)   \right| \ud s + \| \ovl{u} - u^g \|_{L^2_x L^{\infty}_x} \\
& \quad \quad \quad \quad \lesssim \int_0^t \left| \ovl{X}(s,0,x,v) - X^g(s,0,x,v)   \right| \ud s + \sqrt{ M^2 + \epsilon^2}.
\end{align*} Whence, thanks to Gronwall's lemma,
\begin{equation}
\sup_{Q_T} |(\ovl{X} - X^g)(t,0,x,v)| \lesssim \sqrt{ M^2 + \epsilon^2 }.
\label{eq:stabiliteMepsilonenespace}
\end{equation} Using (\ref{eq:characteristics}), we also deduce
\begin{equation}
\sup_{Q_T} |(\ovl{V} - V^g)(t,0,x,v)| \lesssim \sqrt{ M^2 + \epsilon^2 }.
\label{eq:stabiliteMepsilonenvitesse}
\end{equation}


\vspace{0.5em}
\textit{Step 2. Conclusion.} 
Let us choose $\epsilon>0$ and $M>0$ small enough so that 
\begin{equation*}
\sup_{Q_T}\| (\ovl{X}  - X^g)(t,x,v) \| + \sup_{Q_T} \| (\ovl{V}-V^g)(t,x,v) \| < \frac{\delta}{2},
\end{equation*} with the choice (\ref{eq:choiceofdelta}). Thus, combining (\ref{eq:referencecharacteristics}) with (\ref{eq:stabiliteMepsilonenespace}) and (\ref{eq:stabiliteMepsilonenvitesse}), we obtain the following property.
\begin{align}
\textrm{For every } &  (x,v)\in \T^2 \times \R^2, \textrm{ there exists } t_{(x,v)} \in \left[ \frac{T}{48}, \frac{47T}{48}  \right] \textrm{such that} \nonumber \\
& \quad \quad \quad \quad  X^g(t_{(x,v)},0,x,v) \in H_{\frac{\delta}{2}}, \label{eq:Xgdansdeltademi} \\
& \quad \quad \quad \quad  \left| V^g(t_{(x,v)},0,x,v) \cdot n_H \right| > \frac{7}{2}. \label{eq:Vgvitessesuffisante}
\end{align} We shall show that this entails (\ref{eq:characteristicsarerelevant}). \par 

Indeed, let us set 
\begin{equation*}
s_0 := \log(1 + \delta).
\end{equation*} Let $(x,v) \in \T^2 \times \R^2$ and let $t_{(x,v)}$ be given by (\ref{eq:Xgdansdeltademi}). Then, we have the following
\begin{align*}
& \left| \left( X^g(t_{(x,v)},0,x,v) + (1-e^{s_0})V^g(t_{(x,v)},0,x,v)   \right)\cdot n_H       \right| \\
& \quad \quad \quad \quad \quad \geq - |X^g(t_{(x,v)},0,x,v)\cdot n_H| + (e^{s_0}-1)|V^g(t_{(x,v)},0,x,v)\cdot n_H| \\
& \quad \quad \quad \quad \quad \geq -\frac{\delta}{2} + \frac{7}{2}(e^{s_0} - 1) \\
& \quad \quad \quad \quad \quad > 2 \delta, 
\end{align*} using (\ref{eq:Xgdansdeltademi}), (\ref{eq:Vgvitessesuffisante}) and the choice (\ref{eq:choiceofdelta}). Thus, thanks to the previous inequality and (\ref{eq:characteristics}), and denoting $X^g(t):=X(t,0,x,v)$ and $V^g(t) :=V^g(t,0,x,v) $ for $(x,v)$ fixed, we get
\begin{align*}
& \left| \left( X^g(t_{(x,v)}) + (1-e^{s_0}) V^g(t_{(x,v)}) - X^g(t_{(x,v)} - s_0)     \right)\cdot n_H   \right| \\
& = \left| \left( X^g(t_{(x,v)}) + (1-e^{s_0}) V^g(t_{(x,v)}) - X^g(t_{(x,v)} - s_0,t_{(x,v)},(X^g,V^g)(t_{(x,v)})) \right)\cdot n_H   \right|    \\
& = \left| \int_{t_{(x,v)}} ^{t_{(x,v)} - s_0 } \int_{t_{(x,v)}}^{\eta} e^{z - \eta } u^g(z, X^g(z)) \ud z \ud \eta  \right| \\ 
& \leq s_0^{\frac{3}{2}} \|u^g\|_{L^2L^{\infty}_x} \\
& < \delta,
\end{align*} thanks to Cauchy-Schwarz's inequality and (\ref{eq:choiceofdelta}). Let us observe that this entails that 
\begin{equation*}
X^g(t_{(x,v)}- s_0,0,x,v) \not \in H_{\delta}.
\end{equation*}  Combining this with (\ref{eq:Xgdansdeltademi}) and thanks to the intermediate value theorem, we deduce that there exists $\sigma_{(x,v)} \in (t_{(x,v)} - s_0, t_{(x,v)})$ such that
\begin{equation}
X^g(\sigma_{(x,v)},0,x,v) \in \partial H_{\delta}.
\label{eq:sigmaxv}
\end{equation} Let us observe that, thanks to (\ref{eq:choiceofdelta}), 
\begin{equation}
\sigma_{(x,v)} \in \left[ \frac{T}{48} , \frac{47T}{48}  \right], \quad \forall (x,v) \in \T^2 \times \R^2.
\label{eq:sigmaxvontime}
\end{equation} Finally, thanks to (\ref{eq:characteristics}) and (\ref{eq:Vgvitessesuffisante}), we deduce
\begin{align*}
& | V^g(\sigma_{(x,v)},0,x,v) \cdot n_H | \\
& \quad \quad = \left| V^g(\sigma_{(x,v)}, t_{(x,v)}, (X^g,V^g)(t_{(x,v)},0,x,v)) \cdot n_H  \right| \\
& \quad \quad \geq e^{t_{(x,v)} - \sigma_{(x,v)} } |V^g(t_{(x,v)},0,x,v)| - \sqrt{s_0} \|u^g\|_{L^2_t L^{\infty}_x} \\
& \quad \quad \geq 3,
\end{align*} where we have used the Cauchy-Schwarz's inequality and (\ref{eq:choiceofdelta}). Thus, combining last inequality with (\ref{eq:sigmaxv}) and (\ref{eq:sigmaxvontime}), we find (\ref{eq:characteristicsarerelevant}).

\end{proof}

\begin{proof}[Proof of Theorem \ref{thm:controllability}. Controllability part.] \par
Let us choose $\epsilon \leq \min\left\{ \epsilon_0,\epsilon_1 \right\}$, where $\epsilon_0$ is given by Proposition \ref{proposition:fixedpoint}, and $\epsilon_1$ is given by Proposition \ref{proposition:fixedpointisrelevant}. Let us consider $T\geq T_0$, according to Proposition \ref{proposition:referencetrajectory}, and $\tau_1,\tau_2 >0$. Thus, define 
\begin{equation*}
T_f = T + \tau_1 + \tau_2.
\end{equation*} We shall prove that (\ref{eq:nullcontrollability}) holds in a large enough time $T_f >0$, which will be done in three steps. \par

\vspace{0.5em}
\textit{Step 1. From the initial configuration to a confinement in} $\omega$. \par 
By the choice of $\epsilon$, we apply Proposition \ref{proposition:fixedpoint} in time $T>0$ large enough, which provides a fixed point of $\V_{\epsilon}$, that we denote $g^*$ and a strong solution $(g^*,u^{g^*})$ of (\ref{eq:VNS}) for some $G\in \C^0([0,T]\times \T^2 \times \R^2)$. We observe that from the construction of $\V_{\epsilon}$ and (\ref{eq:referencezeroatendandbeginning}), we have
\begin{equation}
\V_{\epsilon}[g^*](T,x,v):= \Pi\left( \tilde{\V}_{\epsilon}[g^*]_{|([0,T]\times (\T^2 \setminus H_{\delta}) \times \R^2)\cup([0,\frac{T}{48}] \times \T^2 \times \R^2)} \right)(T,x,v).
\end{equation} In particular, it comes from the definition of $\Pi$ that 
\begin{align}
& \V_{\epsilon}[g^*](T,x,v) = \tilde{\V}_{\epsilon}[g^*](T,x,v), & \forall (x,v) \in (\T^2 \setminus \omega) \times \R^2, \label{eq:confinement} \\
& \V_{\epsilon}[g^*](0,x,v) = f_0(x,v), &  \forall (x,v) \in \T^2\times \R^2. \label{eq:initialconfiguration} 
\end{align} Moreover, by (\ref{eq:absortion}) and (\ref{eq:explicitsolutionVNS}),
\begin{equation*}
\tilde{\V}_{\epsilon}[g^*](T,x,v) = e^{2T} f_0((X^{g^*},V^{g^*})(0,T,x,v)).
\end{equation*} Hence, since $ \epsilon \leq \epsilon_1$, Proposition \ref{proposition:fixedpointisrelevant} applies, which implies, thanks to the absorption procedure described by (\ref{eq:absortion}) and (\ref{eq:absortionfunction}) that $\tilde{\V}_{\epsilon}[g](T,x,v) =0$ in $(\T^2\setminus \omega) \times \R^2$. Thus, by (\ref{eq:confinement}), we get
\begin{equation}
g^*(T,x,v) = 0, \quad \forall (x,v) \in (\T^2 \setminus \omega) \times \R^2. \label{eq:confinementfinale}
\end{equation} 

\vspace{0.5em}
\textit{Step 2. From the confinement in }$\omega$\textit{ to the zero distribution.}\par 
Let $\zeta \in \C^{\infty}(\R)$ such that $0\leq \zeta \leq 1$ and
\begin{align}
& \zeta(t) = 1, \quad  \forall t\leq 0, \label{eq:focntionrecollementdebut} \\
& \zeta(t) = 0, \quad  \forall t \geq 1. \label{eq:fonctionrecollementfinal}  
\end{align} Then, let us define
\begin{equation*}
f_{\flat}(t,x,v):= \zeta \left(\frac{t}{\tau_1} \right) g^*(T,x,v), \quad \forall (t,x,v) \in (0,\tau_1)\times \T^2 \times \R^2. 
\end{equation*} Thus,
\begin{equation*}
f_{\flat}|_{t=0} = g^*|_{t=T}, \quad \textrm{ and } \quad f_{\flat}|_{\tau_1} = 0,
\end{equation*} We associate to $f_{\flat}$ the velocity field $u_{\flat}$ obtained by solving the associated system (\ref{eq:NavierStokesforg}), which is possible thanks to Proposition \ref{proposition:existenceNSg}, as $f_{\flat}$ has the same regularity in $(x,v)$ as $g^*$ at any time $\tau_1$. We observe that $u_{\flat}(0)=u^{g^*}(T)$ by construction.

\vspace{0.5em}
\textit{Step 3. From the zero distribution to the stationary fluid.} \par 
According to Step 2, after a time $\tau_1$, the field evolves from $u^{g^*}(T)$ to $u_{\flat}(\tau_1)$. Then, Theorem \ref{thm:CoronFursikov} allows to modify the field in order to reach zero in time $\tau_2$. Meanwhile, the distribution function can be modified accordingly, keeping the particles confined in $\omega$. Let us apply Theorem \ref{thm:CoronFursikov} with 
\begin{equation*}
\tau = \tau_2, \quad M= \T^2, \quad M_0 = \omega, \quad \hat{y} \equiv 0, \quad y_0 = u_{\flat}(\tau_1). 
\end{equation*} This provides a control $w_{\sharp}\in \C^{\infty}([0,\tau_2]\times \T^2)$ such that 
\begin{equation}
\supp w_{\sharp} \subset (0,\tau_2) \times \omega
\label{eq:controlconfined}
\end{equation} and such that the solution of
\begin{equation}
\left\{  \begin{array}{ll}
\partial_t u_{\sharp} + u_{\sharp}\cdot \nabla u_{\sharp} - \Delta u_{\sharp} + \nabla p_{\sharp} = w_{\sharp}, & (t,x) \in (0,\tau_2) \times \T^2, \\
\div_x u_{\sharp} = 0, & (t,x) \in (0,\tau_2) \times \T^2, \\
u_{\sharp}|_{t=0} = u_{\flat}(\tau_1), & x \in \T^2,
\end{array} \right.
\end{equation} satisfies 
\begin{equation}
u_{\sharp}|_{t=\tau_2} \equiv 0.
\label{eq:champdevitessefinale}
\end{equation} We thus define the associated distribution function as 
\begin{equation*}
f_{\sharp}(t,x,v) := (\mathcal{Z}_1,\mathcal{Z}_2)(v)\cdot w_{\sharp}(t,x), \quad \forall (t,x,v) \in (0,\tau_2) \times \T^2 \times \R^2,
\end{equation*} where $\mathcal{Z}_1$ and $\mathcal{Z}_2$ are given by (\ref{eq:Z1integral}) and (\ref{eq:Z2integral}). As a consequence of (\ref{eq:controlconfined}) and Step 2,
\begin{equation*}
f_{\sharp}|_{t=0} = f_{\flat}|_{t=\tau_1} = 0, \quad f_{\sharp}|_{t=\tau_2} = 0. 
\end{equation*}

\vspace{0.5em}
\textit{Conclusion.}
We put together these steps to construct a suitable solution. Let us define
\begin{equation}
f(t) := \left\{  \begin{array}{ll}
g^*(t), & t \in [0,T), \\
f_{\flat}(t-T), & t \in [T,T+\tau_1), \\
f_{\sharp}(t-T-\tau_1), & t\in [T+\tau_1, T + \tau_1 + \tau_2], 
\end{array} \right.
\end{equation} and 
\begin{equation}
u(t) := \left\{  \begin{array}{ll}
u^{g^*}(t), & t \in [0,T), \\
u_{\flat}(t-T), & t \in [T,T+\tau_1), \\
u_{\sharp}(t-T-\tau_1), & t\in [T+\tau_1, T + \tau_1 + \tau_2]. 
\end{array} \right.
\end{equation} According to the previous arguments, $(f,u)$ is a strong solution of (\ref{eq:VNS}) and satisfies (\ref{eq:nullcontrollability}).
\end{proof}

\section{Proof of Theorem \ref{thm:controllability}. Uniqueness.}
\label{sec: uniquenessVNS}

The goal of this section is to show that the strong solution of (\ref{eq:VNS}) obtained in Section \ref{sec: fixedpointVNS} is unique within a certain class. \par 

In Corollary \ref{corollary:regularityofg} it is proved that the solution of system (\ref{eq:VNS}) obtained by the fixed-point procedure of Section \ref{sec: fixedpointVNS}, (see Proposition \ref{proposition:fixedpoint}) enjoys some regularity properties. Next result, inspired from \cite[Section 8]{UkaiOkabe}, shows that the solution in this class is unique. 

\begin{proposition}
Let $f_0 \in \C^1(\T^2\times \R^2)$ satisfying (\ref{eq:conditionforuniqueness}) and let $G \in \C^0(Q_T)$. Then, the strong solution of system (\ref{eq:VNS}), according to Definition \ref{definition:strong}, satisfying conditions (\ref{eq:regulariteunicite}) and (\ref{eq:classeunicite}) is unique.
\label{proposition:uniqueness}
\end{proposition}

\begin{proof}
Let $f_1 = \tilde{\V}_{\epsilon}[f_1]$, for $\epsilon \leq \epsilon_0$. Let us suppose that $(f^2,u_2)$ is a strong solution of system (\ref{eq:VNS}) with initial datum $f_0$ and control $G$ and such that (\ref{eq:regulariteunicite}) and (\ref{eq:classeunicite}) are satisfied. \par 
Let $w:=u_1-u_2$, $g:=f^1 - f^2$. We shall use Proposition \ref{proposition:stabilityH32} on the difference $w$. Observe that, thanks to (\ref{eq:conditionforuniqueness}) and (\ref{eq:classeunicite}), we have
\begin{equation*}
\sup_{(t,x,v)\in Q_T} (1 + |v|)^{\gamma + 1} \left( |f^1| + |f^2| + |\nabla_x f^1| + |\nabla_x f^2|    \right) < \kappa'.
\end{equation*} Thus, (\ref{eq:stabilityH32}) yields in this case
\begin{equation*}
\|w\|_{L^2_t H_x^{\frac{3}{2}}} \lesssim \left( \|j_g\|_{L^{\infty}_t L^2_x} +  \| \rho_g \|_{L^{\infty}_t L^2_x} \| u_2\|_{L^2_t L^{\infty}_x}\right).
\end{equation*} Moreover, the Sobolev embedding theorem gives, 
\begin{equation}
\|w\|_{L^2_t L^{\infty}_x} \lesssim \left( \|j_g\|_{L^{\infty}_t L^2_x} + \| \rho_g \|_{L^{\infty}_t L^2_x} \| u_2 \|_{L^2_t L^{\infty}_x}  \right).
\label{eq:majorationchampdevitesses}
\end{equation} On the other hand, we observe that condition (\ref{eq:classeunicite}) gives 
\begin{align*}
& (1+|v|)|\nabla_{x,v}f^2(t,(X^1,V^1)(0,t,x,v))|  \\
& \quad \quad \quad \quad \quad \quad \leq \frac{\kappa' (1+|v|)}{(1+ |V^1(0,t,x,v)|)^{\gamma+1}}  \\
& \quad \quad \quad \quad \quad \quad \leq \frac{C(\kappa',\gamma)}{(1 + |v|)^{\gamma}},  
\end{align*} proceeding in the same fashion as in (\ref{eq:majorationVg}). As a result,
\begin{equation}
\sup_{(t,x) \in \Omega_T}\int_{\R^2} (1+|v|) \left|\nabla_v f^2 (t,(X^1,V^1)(0,t,x,v)\right| \ud v \leq \tilde{C}(\kappa',\gamma),
\label{eq:majorationgradf2}
\end{equation} for some constant $\tilde{C}(\kappa',\gamma)>0$. \par 
Next, we observe that the difference of the distribution functions, $g$, satisfies the following Vlasov equation
\begin{equation*}
\partial_t g + v\cdot \nabla_x g + \div_v\left[  (u_1 - v ) g \right] = -w\cdot \nabla_v f^2, \quad \forall (t,x,v) \in Q_T.
\end{equation*} Consequently, by the method of characteristics, we have
\begin{eqnarray}
|g(t,x,v)| &\leq & e^{2T} \left|\int_0^t w(s,X^1(0,s,x,v))\cdot \nabla_v f^2(s,(X^1,V^1)(0,s,x,v) \ud s  \right| \nonumber \\
& \lesssim & \int_0^t \|w(s,\cdot)\|_{L^{\infty}_x} \left|\nabla_v f^2(s,(X^1,V^1)(0,s,x,v)\right| \ud s. \nonumber
\end{eqnarray} Thus,
\begin{equation*}
(1+|v|)|g(t,x,v)| \leq \int_0^t \|w(s,\cdot)\|_{L^{\infty}_x} (1+|v|)\left|\nabla_v f^2(s,(X^1,V^1)(0,s,x,v)\right| \ud s,
\end{equation*} which implies, thanks to (\ref{eq:majorationgradf2}) and (\ref{eq:majorationchampdevitesses}),
\begin{eqnarray}
\sup_{x \in \T^2 } \left( |j_g(t,x)| + |\rho_g(t,x)|\right) &\lesssim & \int_0^t \|w(s,\cdot)\|_{L^{\infty}_x} \ud s \nonumber \\
& \lesssim &  \int_0^t ( \| j_g(s)\|_{L^2_x} + \|\rho_g(s) \|_{L^2_x} ) \ud s \nonumber \\
& \lesssim &  \int_0^t \sup_{x\in \T^2 } \left( \left|j_g(s)\right| + \left| \rho_g(s)  \right|  \right) \ud s, \quad \forall t\in [0,T], \nonumber
\end{eqnarray} which, by Gronwall's lemma entails, since $\rho_g(0) = j_g(0) = 0$, that
\begin{equation*}
\rho_g (t,x) = 0, \quad j_g(t,x) = 0, \quad \forall (t,x) \in \Omega_T.
\end{equation*} Moreover, we deduce from this that the difference $w (t)= (u_1-u_2)(t)$ satisfies, for every $t \in [0,T]$,
\begin{equation*}
\left\{ \begin{array}{ll}
\partial_t w -\Delta_x w(t) +\nabla_x \pi(t) = -(u_1\cdot \nabla )u_1 + (u_2 \cdot \nabla ) u_2 + \rho_{f^1}(t)w(t), & \Omega_T, \\
\div_x w(t) =0, & \Omega_T, 
\end{array} \right.
\end{equation*} which, according to Theorem \ref{thm:Leraystabilite} must imply that $u_1 = u_2$ in $\Omega_T$. In particular, the characteristics associated to $-v + u_1$ and to $-v + u_2$ coincide. Then, $f^1 = f^2$ in $Q_T$.
\end{proof}

\section{Perspectives and comments}
\label{sec: commentsVNS}

We have proved in Theorem \ref{thm:controllability} a null-controllability result for the Vlasov-Navier-Stokes system in dimension 2. Let us make some comments about the possible limitations of this result. \par 

First of all, we observe that a natural limitation concerning dimension comes from the difficulties presented by the three-dimensional Navier-Stokes system. In particular, since the uniqueness of weak solutions for this system is still unknown, there is no hope, a priori, to obtain better results when considering the coupling with a Vlasov equation. \par  

On the other hand, the Navier-Stokes in dimension 2 allows the use of fine stability estimates and a certain regularising effect, which permits the definition of classical characteristics associated with the velocity field. This is essential to describe the absorption procedure of Section \ref{sec: fixedpointVNS}.\par 

Let us emphasise that the controllability result of Theorem \ref{thm:controllability} allows to control at the same time the distribution function of particles and the motion of the fluid in which they are immersed. This can be done thanks to the return method, by exploiting in a crucial manner the two coupling terms of the system: the coupling term in the Vlasov equation and the drag force term present in the Navier-Stokes system. Furthermore, we can achieve the control of all the components with a scalar control acting only on the Vlasov equation. This kind of feature is already well understood in the case of the Navier-Stokes system \cite{CoronGuerrero,CoronLissy}, where the controlled component is arbitrary. In the Vlasov-Navier-Stokes case, our methods work only if the controlled component is the distribution function.   \par 

The result of Theorem \ref{thm:controllability} can be seen as a kinetic version of the result obtained for a fluid-structure system in \cite{BoulakiaOsses} (see \cite {BoulakiaGuerrero} for a three-dimensional result). Indeed, whereas the fluid-structure problems aims at controlling the trajectory of a macroscopic body immersed in a fluid, the kinetic approach allows to treat the dynamics of a cloud of microscopic particles in a fluid, replacing the individual effects of particles by a mesoscopic description.


\begin{appendix}


\section{Review of the Stokes system}
\label{sec: AppendixS}

Let us recall a well-posedness result for the non-stationary Stokes system, i.e.,
\begin{equation}
\left\{ \begin{array}{ll}
\partial_t u -\Delta_x u + \nabla_x p = f, & (t,x) \in (0,T) \times \T^2, \\
\div_x u = 0, & (t,x) \in (0,T) \times \T^2, \\
u|_{t=0}=u_0, & x\in \T^2,
\end{array} \right.
\label{eq:homogeneousStokes}
\end{equation} due to Y. Giga and H. Sohr (see \cite[Theorem 2.8, p.82]{GigaSohr}). This result gives a very general framework for the $L^s_t L^q_x$ well-posedness, under suitable assumptions on the data. Particularly, in the case of initial data, we are lead to the following spaces (see \cite[p.77]{GigaSohr}). Let 
\begin{equation*}
D^{\alpha,s}_{A,q} := \left\{ u \in L^q(\T^2); \, \|u\|_{L^q} + \left( \int_0^{\infty} \| t^{1-\alpha} A e^{-tA} u \|_{L^q} \frac{\ud t}{t} \right)^{\frac{1}{s}} < \infty  \right\},
\end{equation*} where $\alpha \in (0,1)$, $s \in [1,\infty)$, and $A$ is the Stokes operator $A = \mathcal{P}(-\Delta)$, $\mathcal{P}$ being the Leray projector.  In this context, the result is as follows.

\begin{thm}[\cite{GigaSohr}]
Let $1<s<\infty$ and $1<q<\infty$. For any $f\in L^s(0,T;L^q(\T^2)^2)$ and $u_0 \in D_{A,q}^{1-\frac{1}{s},s}$, there exists a unique solution of system (\ref{eq:homogeneousStokes}), satisfying $u\in L^s(0,T;W^{2,q}(\T^2)^2)$ and $\partial_t u \in L^s(0,T;L^q(\T^2)^2)$. Moreover, there exists a constant $C_0>0$ such that 
\begin{equation}
\int_0^T \|\partial_t u(t) \|^s_{L^q} \ud t + \int_0^T \|D^2 u \|_{L^q}^s \ud t \leq C \left(  \int_0^T \|f\|_{L^q}^s \ud t + \|u_0\|_{D_{A,q}^{1-\frac{1}{s},s}} \right).
\label{eq:regularityStokes}
\end{equation}
\label{thm:Stokesclassical}
\end{thm}


\section{Review of the Navier-Stokes system on the 2-dimensional torus}
\label{sec: AppendixNS}

We shall need to use some classical results on the Navier-Stokes, that we gather here for reference. \par 
Let us consider $F\in L^2(0,T;\V')$ and $u_0\in \Hdiv$ and the Navier-Stokes system
\begin{equation}
\left\{  \begin{array}{ll}
\partial_t u + \left( u\cdot \nabla \right)u -\Delta u + \nabla p = F, & (t,x) \in (0,T) \times \T^2, \\
\div u(t,x) = 0, & t\in (0,T), \\  
u_{|t=0} = u_0, & x\in T^2.
\end{array} \right.
\label{eq:NS}
\end{equation} Following \cite[p.42]{Chemin}, we shall use the following notion of solution.

\begin{definition}
A time-dependent vector field $u$ is a solution of (\ref{eq:NS}) whenever
\begin{equation}
u \in \C^0([0,T];\V_{\sigma}') \cap L^{\infty}(0,T;\Hdiv) \cap L^2(0,T;\V_{\sigma})
\end{equation} and for any $\psi \in \C^1([0,T];\V_{\sigma})$ and $t\in (0,T]$, one has 
\begin{align}
& \int_{\T^2} u(t)\psi(t) \ud x + \int_0^t \int_{\T^2} \nabla u : \nabla \psi - u \otimes u \cdot \nabla \psi -u\partial_t \psi \ud s \ud x \nonumber \\
& \quad \quad \quad \quad = \int_{\T^2} u_0 \psi(0) \ud x + \int_0^t \langle F(s), \psi(s) \rangle_{\V' \times \V} \ud s. \label{eq:formulationfaibleNS}
\end{align} 
\label{definition:Leraysolution}
\end{definition} We have the following classical result, due to J. Leray (see \cite[Theorem 2.3]{Chemin}).

\begin{thm}
Let $F\in L^2(0,T;\V')$ and $u_0 \in \Hdiv$. There exists a unique global solution of (\ref{eq:NS}), i.e., for any $T>0$, in the sense of Definition \ref{definition:Leraysolution}. Moreover, this solution satisfies the energy estimate
\begin{equation}
\|u(t)\|^2_{L^2} + \int_0^t \|\nabla u(s)\|^2_{L^2} \ud s \leq e^t \left( \|u_0\|_{L^2}^2 + \int_0^t \|F(s)\|_{\V'}^2 \ud s \right).
\label{eq:NSestimate}
\end{equation}
\label{thm:Leray}
\end{thm} In the 2-dimensional case, we also have some stability estimates (see \cite[Theorem 3.2, p.56]{Chemin}).

\begin{thm} Let $u$ and $v$ be two solutions of system (\ref{eq:NS}) associated to $(u_0,F)$ and $(v_0,G)$ respectively. Then,
\begin{align}
& e^{-t} \| (u-v)(t) \|_{L^2}^2 + \int_0^t \| \nabla(u-v)(s) \|_{L^2}^2 \ud s \nonumber \\
& \quad \quad \quad \leq e^{CE^2(t)} \left( \|u_0 - v_0 \|_{L^2}^2 + \int_0^t \|(F-G)(s) \|_{\V_{\sigma}'} \ud s  \right),
\end{align} with 
\begin{equation*}
E(t):= e^t\min\left\{ \|u_0\|_{L^2}^2 + \int_0^t \|F(s)\|_{\V'}^2 \ud s, \|v_0\|_{L^2}^2 + \int_0^t \|G(s)\|_{\V'}^2 \ud s        \right\}.
\end{equation*}
\label{thm:Leraystabilite}
\end{thm} Moreover, in the specific case of the 2-dimensional torus, we have the following result of propagation of regularity (see \cite[Theorem 3.7, p. 80]{Chemin}.

\begin{thm}
Let $u_0 \in H^{\frac{1}{2}}_0(\T^2) \cap \Hdiv$ and $F \in L^2(0,T; H^{-\frac{1}{2}}_0(\T^2))$. Then the unique solution of (\ref{eq:NS}) satisfies
\begin{equation}
u \in \C^0([0,T];H^{\frac{1}{2}}_0(\T^2)) \cap L^2(0,T;H^{\frac{1}{2}}_0(\T^2))  
\label{eq:NSregularite}
\end{equation}  and 
\begin{align}
& \|u(t,x)\|^2_{H^{\frac{1}{2}}_0} + \int_0^t \|\nabla u(s)\|^2_{H^{\frac{1}{2}}_0} \ud s \label{eq:NSestimate32} \\
& \quad \quad \quad \leq exp \left( c\int_0^t \|\nabla u(s) \|_{L^2}^2 \ud s \right) \left( \|u_0\|_{H^{\frac{1}{2}}_0}^2 + \int_0^t \|F(s)\|_{H^{-\frac{1}{2}}_0}^2 \ud s \right). \nonumber 
\end{align}
\label{thm:Lerayregularite}
\end{thm}

\end{appendix}

\bibliographystyle{plain}                            

\begin{thebibliography}{}

\end{thebibliography}


\begin{thebibliography}{99}

\bibitem{Asano} K. Asano. On local solutions of the Cauchy problem for the Vlasov-Maxwell Equation. \emph{ Commun. Math. Phys.} vol.106, pp. 551-568. 1986. 

\bibitem{BLR} C. Bardos, G. Lebeau and J. Rauch. Sharp Sufficient Conditions for the Observation, Control, and Stabilization of Waves from the Boundary. \emph{SIAM J. Control Optim.} vol. 30:5, pp. 1024–1065. 1992.

\bibitem{BoudinDesvillettes} L. Boudin, L. Desvillettes, C. Grandmont and A. Moussa. Global existence of solutions for the coupled Vlasov and Navier-Stokes equations. \emph{Diff. and Int. Eq.}, vol. 22:11-12, 2009.

\bibitem{Boudin} L. Boudin, C. Grandmont, A. Lorz and A. Moussa. Modelling and numerics for respiratory aerosols. To appear in \emph{Comm. in Comp. Physics}. 2016.

\bibitem{BoulakiaOsses} M. Boulakia and A. Osses. Local null controllability of a two-dimensional fluid-structure interaction problem. \emph{ESAIM Control Optim. Calc. Var.} vol. 14:1. pp. 1-42. 2008.

\bibitem{BoulakiaGuerrero} M. Boulakia and S. Guerrero. Local null controllability of a fluid-solid interaction problem in dimension 3. \emph{J. European Math. Society}, vol. 15:3. pp. 825-856. 2013. 

\bibitem{Boyer} F. Boyer and P. Fabrie. Mathematical tools for the study of incompressible Navier-Stokes equations and related models. \emph{Springer series in Appl. Math,} vol. 143. 2013. 

\bibitem{Chemin} J.-Y. Chemin, B. Desjardins, I. Gallagher and E. Grenier. Mathematical Geophysics. Clarendon Press, Oxford. 2006.

\bibitem{ChoiKwon} Y.-P. Choi and B. Kwon. Global well-posedness and large-time behavior for the inhomogeneous Vlasov-Navier-Stokes equations. \emph{Nonlinearity}. vol. 28. pp. 3309-3336. 2015.

\bibitem{Choi} Y.-P. Choi. Large-time behavior for the Vlasov/compressible Navier-Stokes equations, \emph{Journal of Mathematical Physics,} Vol. 57:7, 071501. 2016.

\bibitem{CoronFursikov} J.-M. Coron and A. Fursikov. Global exact controllability of the 2D Navier-Stokes equations on a manifold without boundary. \emph{Russian Journal of Mathematical Physics}, vol.4:4, p. 429-448, 1996.

\bibitem{Coron} J.-M. Coron. Control and Nonlinearity. Mathematical Surveys and Monographs, vol. 136. American Mathematical Society, 2007.

\bibitem{CoronGuerrero} J.-M. Coron and S. Guerrero. Local null controllability of the two-dimensional Navier-Stokes system in the torus with a control force having a vanishing component. \emph{J. Math. Pures et Appl.} Vol. 92:9. pp. 528-545. 2009.

\bibitem{CoronLissy} J.-M. Coron and P. Lissy. Local null controllability of the three-dimensional Navier-Stokes system with a distributed control having two vanishing components. \emph{ Invent. Math.} Vol. 198. pp. 833–880. 2015.

\bibitem{Crippa}G. Crippa. The flow associated to weakly differentiable vector fields. \emph{Theses of Scuola Normale Superiore di
Pisa,} Vol. 12. Edizioni della Normale, Pise, 2009.


\bibitem{DesvillettesGolse} L. Desvillettes, F. Golse and V. Ricci. The mean-field limit for solid particles in a Navier-Stokes flow. \emph{J. Statistical Physics.} vol. 131:5, pp. 941-967. 2008.

\bibitem{Evans} L. C. Evans. Partial Differential Equations. \emph{Graduate Studies in Mathematics. AMS}, vol. 19, 1998.

\bibitem{GigaSohr} Y. Giga and H. Sohr. Abstract $L^p$ estimates for the Cauchy problem with applications to the Navier-Stokes equations in exterior domains. \emph{J. Funct. Anal.} Vol.102:1, pp. 72–94. 1991.


\bibitem{Goudon1} T. Goudon, P.E. Jabin and A. Vasseur. Hydrodynamic limits for Vlasov-Stokes equations: Part I: Light Particles Regime. \emph{ Indiana Univ. Math. J.}, vol.53, pp.1495--1513. 2004.


\bibitem{Goudon2} T. Goudon, P.E. Jabin and A. Vasseur. Hydrodynamic limits for Vlasov-Stokes equations: Part II: Fine Particles Regime,  \emph{ Indiana Univ. Math. J.}, vol.53, pp.1517--1536. 2004.

\bibitem{Hamdache} K. Hamdache. Global existence and large time behaviour of solutions for the Vlasov-Stokes equations. \emph{Japan J. Industr. Appl. Math.} vol. 15. pp. 51-74. 1998.

\bibitem{Glass} O. Glass. On the controllability of the Vlasov-Poisson system. \emph{J. Diff. Eq.} vol. 195, pp. 332-379. 2003.

\bibitem{GlassBourbaki} O. Glass. La m\'{e}thode du retour en contr\^{o}labilit\'{e} et ses applications en m\'{e}canique des fluides (d'apr\`{e}s J.-M. Coron et al.) \emph{S\'{e}minaire Bourbaki}. Novembre 2010.

\bibitem{GDHK1} O. Glass and D. Han-Kwan. On the controllability of the Vlasov-Poisson system in the presence of external force fields. \emph{J. Diff. Eq.} vol. 252. pp. 5453-5491. 2012.

\bibitem{GDHK2} O. Glass and D. Han-Kwan. On the controllability of the relativistic Vlasov-Maxwell system. \emph{J. Math. Pures et Appl.} vol. 103. pp. 695-740. 2015.

\bibitem{VS} I. Moyano. On the controllability of the 2-D Vlasov-Stokes system. \emph{Preprint}. 2015.

\bibitem{UkaiOkabe} S. Ukai and T. Okabe. On classical solutions in the large in time of two-dimensional Vlasov's equation. \emph{Osaka J. Math.} vol. 15:2. pp. 245-261. 1978.

\bibitem{Wollman} S. Wollman. Local existence and uniqueness theory of the Vlasov-Maxwell theorem. \emph{ J. Math, Anal. and Appl.} vol. 127. pp. 103-121. 1987.

\end{thebibliography}

\end{document}